\documentclass[reqno,11pt]{amsart}

\usepackage[a4paper,left=23mm,right=23mm,top=30mm,bottom=30mm,marginpar=25mm]{geometry}
\usepackage{amsmath}
\usepackage{amssymb}
\usepackage{amsthm}
\usepackage{amscd}
\usepackage{mathtools}
\usepackage{url}
\usepackage[ansinew]{inputenc}
\usepackage{color}
\usepackage[final]{graphicx}
\usepackage{esint} 
\usepackage{bbm}
\usepackage{subfigure}
\usepackage{tikz}
\usepackage{cite}
\renewcommand{\epsilon}{\varepsilon}

\numberwithin{equation}{section}

\newtheoremstyle{thmlemcorr}{10pt}{10pt}{\itshape}{}{\bfseries}{.}{10pt}{{\thmname{#1}\thmnumber{ #2}\thmnote{ (#3)}}}
\newtheoremstyle{thmlemcorr*}{10pt}{10pt}{\itshape}{}{\bfseries}{.}\newline{{\thmname{#1}\thmnumber{ #2}\thmnote{ (#3)}}}
\newtheoremstyle{defi}{10pt}{10pt}{\itshape}{}{\bfseries}{.}{10pt}{{\thmname{#1}\thmnumber{ #2}\thmnote{ (#3)}}}
\newtheoremstyle{remexample}{10pt}{10pt}{}{}{\bfseries}{.}{10pt}{{\thmname{#1}\thmnumber{ #2}\thmnote{ (#3)}}}
\newtheoremstyle{ass}{10pt}{10pt}{}{}{\bfseries}{.}{10pt}{{\thmname{#1}\thmnumber{ A#2}\thmnote{ (#3)}}}

\theoremstyle{thmlemcorr}
\newtheorem{theorem}{Theorem}
\numberwithin{theorem}{section}
\newtheorem{lemma}[theorem]{Lemma}
\newtheorem{corollary}[theorem]{Corollary}
\newtheorem{proposition}[theorem]{Proposition}

\theoremstyle{thmlemcorr*}
\newtheorem{theorem*}{Theorem}
\newtheorem{lemma*}[theorem]{Lemma}
\newtheorem{corollary*}[theorem]{Corollary}
\newtheorem{proposition*}[theorem]{Proposition}
\newtheorem{problem*}[theorem]{Problem}
\newtheorem{conjecture*}[theorem]{Conjecture}

\theoremstyle{defi}
\newtheorem{definition}[theorem]{Definition}

\theoremstyle{remexample}
\newtheorem{remarkx}[theorem]{Remark}
\newenvironment{remark}
  {\pushQED{\qed}\remarkx}
  {\popQED\endremarkx}

\theoremstyle{ass}

\newcommand{\Fcal}{\mathcal{F}}

\DeclareMathOperator{\esssup}{ess\,sup}

\newcommand{\normb}[1]{\bigl\|#1\bigr\|}

\newcommand{\N}{\mathbb{N}}
\newcommand{\R}{\mathbb{R}}

\newcommand{\weakly}{\rightharpoonup}
\newcommand{\weaklystar}{\overset{*}\rightharpoonup}

\newcommand{\eps}{\epsilon}
\newcommand{\phirho}{\varphi^{\rho}}
\newcommand{\phirhotilde}{\tilde{\varphi}^{\rho}}

\def\Xint#1{\mathchoice
{\XXint\displaystyle\textstyle{#1}}%
{\XXint\textstyle\scriptstyle{#1}}%
{\XXint\scriptstyle\scriptscriptstyle{#1}}%
{\XXint\scriptscriptstyle\scriptscriptstyle{#1}}%
\!\int}
\def\XXint#1#2#3{{\setbox0=\hbox{$#1{#2#3}{\int}$}
\vcenter{\hbox{$#2#3$}}\kern-.5\wd0}}
\def\dashint{\,\Xint-}

\DeclarePairedDelimiter\abs{\lvert}{\rvert}
\DeclarePairedDelimiter{\norm}{\lVert}{\rVert}

\newcommand{\Rn}{\R^{n}}

\renewcommand{\phi}{\varphi}

\newcommand{\dx}{\, dx}
\newcommand{\dy}{\, dy}

\newcommand{\F}{\mathcal{F}}

\newcommand{\Ccr}{C_{c}^{\infty}(\R^n)}

\newcommand{\Sinf}{S^{\alpha,\infty}}

\newcommand{\Sp}{S^{\alpha,p}}
\newcommand{\Spo}{S^{\alpha,p}_0}
\newcommand{\Spg}{S^{\alpha,p}_g}

\newcommand{\starto}{\stackrel{*}{\rightharpoonup}}
\newcommand{\weakto}{\rightharpoonup}
\newcommand{\na}{\nabla^{\alpha}}
\newcommand{\da}[1]{(-\Delta)^{#1}}
\newcommand{\nanl}{\nabla^{\alpha}_{\normalfont{\text{NL}}}}

\def\Xint#1{\mathchoice
   {\XXint\displaystyle\textstyle{#1}}%
   {\XXint\textstyle\scriptstyle{#1}}%
   {\XXint\scriptstyle\scriptscriptstyle{#1}}%
   {\XXint\scriptscriptstyle\scriptscriptstyle{#1}}%
   \!\int}
\def\XXint#1#2#3{{\setbox0=\hbox{$#1{#2#3}{\int}$}
     \vcenter{\hbox{$#2#3$}}\kern-.5\wd0}}

\def\dashint{\Xint-}
\newcommand{\dac}{(-\Delta)^{\frac{1-\alpha}{2}}}
\newcommand{\Spmn}{S^{\alpha,p}(\R^n;\R^m)}

\newcommand{\Rmn}{\mathbb{R}^{m \times n}}

\newcommand{\Spom}{S^{\alpha,p}_{0}(\Omega;\R^m)}
\newcommand{\Spgm}{S^{\alpha,p}_{g}(\Omega;\R^m)}

\newcommand{\Sinfperm}{S^{\alpha,\infty}_{per}(Q;\R^m)}
\newcommand{\Winfperm}{W^{1,\infty}_{per}(Q;\R^m)}
\newcommand{\Sinfper}{S^{\alpha,\infty}_{per}(Q)}
\newcommand{\Winfper}{W^{1,\infty}_{per}(Q)}

\newcommand{\Frel}{\mathcal{F}^{\mathrm{rel}}_{\alpha}}

\newcommand{\Lipb}{\mathrm{Lip}_{b}}
\newcommand{\Lip}{\mathrm{Lip}}
\makeatletter
\g@addto@macro\bfseries{\boldmath}
\makeatother

\title[]{Quasiconvexity in the fractional calculus of variations: Characterization of lower semicontinuity and relaxation}

\author{Carolin Kreisbeck}
\address{Mathematisch-Geographische Fakult\"at, Katholische Universit\"at Eichst\"att-Ingolstadt, Osten\-stra{\ss}e 28, 85072 Eichst\"att}
\email{carolin.kreisbeck@ku.de}

\author{Hidde Sch\"{o}nberger}
\address{Mathematical Institute, Utrecht University, Postbus 80010, 3508 TA Utrecht, The Netherlands}
\email{h.m.j.schonberger@students.uu.nl}
\begin{document}

%%%%%%%%%%%%%%%%%%%%%%%%%%%% ABSTRACT %%%%%%%%%%%%%%%%%%%%%%%%%%%%%%%%%%%

\maketitle

\thispagestyle{empty}
\begin{abstract} 
Based on recent developments in the theory of fractional Sobolev spaces, an interesting new class of nonlocal variational problems has emerged in the literature.  
These problems, which are the focus of this work, involve integral functionals that depend on Riesz fractional gradients instead of ordinary gradients and are considered subject to a complementary-value condition.
 
With the goal of establishing a comprehensive existence theory, we provide a full characterization for the weak lower semicontinuity of these functionals under suitable growth assumptions on the integrands.
In doing so, we surprisingly identify quasiconvexity, which is  intrinsic to the standard vectorial calculus of variations, as the natural notion also in the fractional setting.  
In the absence of quasiconvexity,
we determine a representation formula for the corresponding relaxed functionals, obtained via partial quasiconvexification outside the region where complementary values are prescribed. Thus, in contrast to classical results,
 the relaxation process induces a structural change in the functional, turning the integrand from a homogeneous into an inhomogeneous one. 
Our proofs rely crucially on an inherent relation between classical and fractional gradients, which we extend to Sobolev spaces, enabling us to transition between the two settings.

\vspace{8pt}

\vspace{8pt}

\vspace{8pt}

\noindent\textsc{Date:} \today.
\end{abstract}
%
%%%%%%%%%%%%%%%%%%%%%%%%%%%% MAIN PART %%%%%%%%%%%%%%%%%%%%%%%%%%%

\section{Introduction}
The modern methods in the calculus of variations revolve, generally speaking, around an existence theory for minimizers of functionals defined on function spaces. 
An essential task
when it comes to applying these methods is to identify criteria to characterize the lower semicontinuity of the functionals under consideration with respect to a suitable topology that guarantees coercivity.
In case these conditions fail, minimizers may not exist, and relaxation techniques are commonly
deployed in order to capture the (asymptotic) behavior of minimizing sequences.

The most prominent and widely studied class of variational problems during the last decades involve integral functionals with integrands depending on functions and their gradient fields, often minimized subject to fixed boundary values. That is, when formulated on Sobolev spaces, one considers functionals of the form
\begin{equation}
\F(v)=\int_{\Omega} f(x, v(x), \nabla v(x))\dx \qquad \text{for $v\in g+ W_0^{1,p}(\Omega;\R^m)$,}
\label{eq:functional_gradient}
\end{equation} 
where $\Omega\subset\R^n$ is an open and bounded set, $p\in(0,1)$, $f:\Omega\times \R^m\times \R^{m\times n}\to \R$ a Carath\'eodory function with suitable growth 
assumptions, and $g\in W^{1,p}(\Omega;\R^m)$. 

It is well-known that quasiconvexity (\'a la Morrey~\cite{Mor}) is the generalized notion of convexity inherently associated with these (vectorial) variational problems; in the sense that $\Fcal$ is weak lower semicontinuous if and only if the integrand $f$ is 
quasiconvex in its third variable~\cite{Mor, Dacorogna, AcerbiFusco}, and that the relaxed version of $\Fcal$, i.e., its weak lower semicontinuous envelope, is obtained via quasiconvexification of $f$~\cite{Dacorognarelaxation, Dacorogna}.
 
The last years have seen increasing efforts to extend these results towards new types of variational functionals, and especially those with nonlocal features have attracted attention.  
On the one hand, this is motivated by applications,
where it is advantageous to take
global effects and long-range interactions into account, like for instance, in models of continuum mechanics, including peridynamics~\cite{Sil00, peridynamics2, MeD15} and new energetic approaches to hyperelasticity~\cite{Bellido}, or when using nonlocal regularizers
in imaging~\cite{OsherGilboa, AubertKornprobst} and machine learning applications~\cite{AntilKhatri, HollerKunish}.
On the other hand, from a theoretical perspective,  variational problems with nonlocalities bring about interesting mathematical challenges; since the standard methods, often based on localization arguments, are not readily applicable.
These difficulties are nicely illustrated for nonlocal double-integral functionals, which have been studied amongst others in~\cite{Pedregalperidynamics,Bellidoperidynamics,Kreisbeckintegral, MCT20}. 
For these functionals, separate convexity arises as the relevant notion for the characterization of lower semicontinuity, but 
it turns out that their relaxation is in general not structure preserving
and no universally applicable representation formula is currently known.
 
A different type of nonlocality involves the well-known concepts of fractional derivatives. 
While the fractional partial differential equations has been a very active field of research for decades, efforts towards developing a well-rounded existence theory for a fractional variational calculus
in multiple dimensions are more recent~\cite{Shieh1,Shieh2,Bellido}.
For an overview of variational problems involving one-dimensional fractional derivatives, we refer e.g.~to~\cite{Torres1}. \medskip

In this paper, we focus on a new class of variational problems, which can be viewed as the most natural fractional analogue of~\eqref{eq:functional_gradient}. It was first proposed in the scalar setting ($m=1$) by Shieh \& Spector in~\cite{Shieh1} and very recently formulated in a vectorial version by Bellido, Cueto \& Mora-Corral~\cite{Bellido}. 
The essential differences when compared to~\eqref{eq:functional_gradient} is that now the Riesz fractional gradient takes the role of the ordinary (weak) gradient. As recently shown by Silhavy in~\cite{Silhavy}, the former is (up to a constant) the unique fractional derivative operator to satisfy natural invariance and scaling properties, and can thus be considered the canonical choice of fractional derivative.

Let $\alpha\in (0,1)$. The Riesz fractional gradient is the singular-integral operator defined for $\varphi \in \Ccr$ by
\[
\na \phi(x):=\mu_{n,\alpha}\int_{\R^n}\frac{\phi(y)-\phi(x)}{\abs{y-x}^{n+\alpha}}\frac{y-x}{\abs{y-x}}\dy, \quad x\in \R^n,
\]
with a specific real constant $\mu_{n,\alpha}$.
Via a fractional integration by parts formula this definition can be extended in a distributional sense to functions $u\in L^p(\R^n)$ with $p\in [1, \infty]$, and componentwise to vector fields $u\in L^p(\R^n;\R^m)$. More details can be found in~Section~\ref{sec:rieszpotential}, besides, for broader context, we refer to~\cite{Obstacle,Mengesha,ShiehRegularity,Schikorra,Spector,Unified} for selected examples of recent research on fractional partial differential equations involving Riesz fractional gradients. 

Replacing the common weak gradient with the Riesz fractional gradient in~\eqref{eq:functional_gradient} necessitates two further changes in the new problem set-up, namely, the proper choice of function spaces, here,
\begin{align}\label{Salphap}
\Sp(\R^n;\R^m)=\{u \in L^p(\R^n;\R^m) \,: \, \na u \in L^p(\R^n;\R^{m\times n})\},
\end{align}
as well as, a substitute for the prescribed boundary values, which are local objects by nature, in terms of complementary values, see~\eqref{Sgalphap}.

Hence, we investigate here functionals of the form
\begin{equation}
\F_\alpha(u) =\int_{\R^n} f(x, u(x), \na u(x))\dx \qquad\text{for} \ u \in \Spg(\Omega;\R^m),
\label{eq:functional1}
\end{equation}
where $\Omega\subset\R^n$ is an open and bounded set, $p\in (1,\infty)$, $f:\R^n \times \R^m\times\Rmn \to \R$ a Carath\'eodory function with suitable growth assumptions, and $g\in \Sp(\R^n;\R^m)$; the complementary-value space $\Spg(\Omega;\R^m)$ is defined as
\begin{align}\label{Sgalphap}
\Spg(\Omega;\R^m)=\{u \in \Sp(\R^n;\R^m) \, : \, u=g \ \text{a.e. in} \ \Omega^c\}.
\end{align}

In light of the new class of variational problems with the functionals from~\eqref{eq:functional1}, it is an important task to develop a comprehensive existence theory for solutions thereof; noticing also, that minimizers of \eqref{eq:functional1} satisfy a (weak) Euler-Lagrange 
%equation, consisting of a
system of fractional partial differential equations \cite{Shieh2,Bellido}.
Our paper contributes two central aspects to this program, a full characterization for the weak lower semicontinuity and a relaxation result, which we 
state in Theorem~\ref{th:characterization} and Theorem~\ref{theo:relaxation} below.

Despite the apparent structural resemblance between $\Fcal_\alpha$ in~\eqref{eq:functional1} and $\Fcal$ in~\eqref{eq:functional_gradient},  a major difference for the technical treatment is rooted in the intrinsic nonlocal structure of the functionals, which enters through the appearance of the fractional derivatives.  Also, the fact that the functions in $\Spg(\Omega;\R^m)$ are defined on an unbounded domain, namely on all of $\R^n$, imposes technical subtleties beyond the standard theory for problems as in~\eqref{eq:functional_gradient}.

Our approach builds on the realization that there is a way to translate between the standard and the fractional problem. To do so, we 
exploit the connection between the gradients of classical and fractional Sobolev functions, using the Riesz potential $I_{\alpha}$ and fractional Laplacian $\da{\alpha/2}$; more explicitly, we prove a suitable extension of the identities for smooth test functions $\phi \in \Ccr$,
\[
\na \varphi=\nabla I_{1-\alpha} \varphi\qquad \text{and}\qquad \nabla \varphi= \na \dac \phi,
\]
to the classical and fractional Sobolev spaces. This allows switching between the two settings and thus, creates a powerful tool that enables us to recourse to the rich pool of results from the by now well-established theory on classical variational problems.
In a  similar spirit, other problems in the theory of partial differential equations have been approached in this way, for example in ~\cite{ShiehRegularity}. 

It becomes evident that our approach, and more generally, an  analysis of the fractional integral functionals in~\eqref{eq:functional1}, would not be possible without the substantial insights into the properties of spaces in~\eqref{Salphap}. In this sense, the works of \cite{Shieh1, Shieh2}, as well as the series of papers \cite{Comi1,Comi2,Comi3}, where Comi \& Stefani and coauthors (besides their primary focus on fractional spaces of bounded variations) introduce and study the fractional Sobolev spaces $\Sp(\R^n)$ from a distributional viewpoint, lay important groundwork. Among the many desirable properties that the new fractional Sobolev spaces share 
with the classical Sobolev spaces we mention for instance, fractional Poincar{\'e}-type, Sobolev and Morrey inequalities,  compact embeddings, and density of smooth functions with compact support; for more context, we refer to Section~\ref{sec:preliminaries}.
\medskip

As announced already, our first main result is a full characterization for the weak lower semicontinuity of $\F_\alpha$, which is given in terms of the quasiconvexity (in the third variable) of the integrand $f$ restricted to $\Omega$.  
As such, it generalizes the work in~\cite{Shieh1, Bellido}, where the convexity (in the scalar case) and polyconvexity  (in the vectorial case) of $f$ are identified as the sufficient conditions, respectively.
\begin{theorem}[Characterization of weak lower semicontinuity]
\label{th:characterization} 
Let $\alpha \in (0,1)$, $p \in (1,\infty)$, $\Omega \subset \R^n$ open and bounded, and $g \in \Sp(\R^n;\R^m)$. Suppose that $f:\R^n\times \R^m\times\Rmn \to \R$ is a Carath\'{e}odory function that satisfies
\[
0 \leq f(x,z,A) \leq a(x)+C(\abs{z}^p+\abs{A}^p)\quad \text{for a.e.~$x\in \R^n$ and for all} \ (z, A)\in \R^m\times \Rmn
\]
with $a \in L^1(\R^n)$ and $C>0$. Then the functional
\[
\F_{\alpha}(u)=\int_{\R^n}f(x,u,\na u)\dx, \quad u\in \Spgm,
\]
is (sequentially) weakly lower semicontinuous on $\Spgm$ if and only if $A \mapsto f(x,z,A)$ is quasiconvex for a.e. $x \in \Omega$ and all $z \in \R^m$.
\end{theorem}

There are two notable aspects about this result we wish to elaborate on. 
First,
the statement of Theorem~\ref{th:characterization} reveals that the necessary and sufficient condition on $f$ is in fact independent of the fractional parameter $\alpha$, and even more, identical (inside of $\Omega$) with the condition that characterizes the weak lower semicontinuity in the standard gradient setting. 
Indeed, this is a consequence of the parallels between the two settings that allow us to utilize the classical weak lower semicontinuity result. 

If asked a priori to make a guess about the correct characterizing condition in Theorem~\ref{th:characterization},   expecting a truly fractional notion does not seem unnatural, especially with an eye to the formal analogy of the problem with~\eqref{eq:functional_gradient}. To reconcile this intuition with our findings of Theorem~\ref{th:characterization}, we introduce the notion of $\alpha$-quasiconvexity as follows: A  function $h: \Rmn \to \R$ is called $\alpha$-quasiconvex if for any $A\in \Rmn$,
\[
h(A)\leq \int_{(0,1)^n}h(A+\na \phi)\dx \qquad\text{for all } \phi \in S^{\alpha, \infty}_{per}((0,1)^n;\R^m),
\]
where $S^{\alpha,\infty}_{per}((0,1)^n;\R^m)$ contains all $(0,1)^n$-periodic functions in $S^{\alpha, \infty}(\R^n;\R^m)$. 
In fact, $\alpha$-quasi\-convexity of $h$ is independent of $\alpha \in (0,1)$, and equivalent with quasiconvexity, as we conclude in~Corollary~\ref{cor:equivalence}.

Second, we point out that Theorem~\ref{th:characterization} requires quasiconvexity only inside of $\Omega$. 
The reason for that is the following, somewhat surprising, property of complementary-value spaces $\Spgm$. In fact, for weakly converging sequences $u_j \weakto u$ in $\Spgm$, we show that the fractional gradients converge strongly in the complement of $\Omega$, see Lemma~\ref{le:strongoutside} below.

The intrinsic difficulty in the proof of Theorem~\ref{th:characterization} is that the link between the classical and fractional gradient goes through a nonlocal operation, which does not preserve complementary values. This necessitates cut-off arguments with careful error estimates for the fractional gradients, see Lemma~\ref{le:leibnizp}. Furthermore, as a replacement for the use of affine functions in the classical setting, we construct in~Lemma~\ref{le:construction} smooth functions with compact support that have a specific fractional gradient at a given point. \medskip

Our second main theorem is, to the best of our knowledge, the first relaxation result for a class of (vectorial) fractional integrals. 
It is stated for functionals $\Fcal_\alpha$ with a homogeneous integrand, that is, no direct dependence on $x$ and $u$, but only on the fractional gradient $\na u$. 
As one would expect in view of the characterization result of Theorem~\ref{th:characterization} the relaxation of $\Fcal$, meaning, its weakly lower semicontinuous envelope, can be obtained by taking the quasiconvexification $f^{\rm qc}$ of the integrand $f$ inside $\Omega$, while $f$ is kept unchanged in the region where complementary values are prescribed. We observe 
the remarkable effect that an integral functional with a homogeneous integrand is turned, through the relaxation process, into one with an inhomogeneous integrand.

\begin{theorem}[Relaxation formula]\label{theo:relaxation}
Let $\alpha \in (0,1)$, $p \in (1,\infty)$, $\Omega \subset \R^n$ open and bounded, and $g \in \Sp(\R^n;\R^m)$. Consider the functional 
\begin{equation}\label{eq:integralfunctional2}
\F_\alpha(u)=\int_{\R^n}f(\na u)\dx, \quad  \ u \in \Spgm,
\end{equation} 
where $f:\R^{m\times n}\to \R$ is continuous and satisfies 
\[
c\abs{A}^p\leq f(A) \leq C\abs{A}^p \quad \text{for all $A\in \R^{m\times n}$}
\]
with constants $C\geq c>0$. 
Then the relaxation of $\F_{\alpha}$ with respect to the weak convergence in $\Spgm$
is given  by
\begin{align*}
\F_\alpha^{\rm rel}(u)&=\inf\{\liminf_{j\to \infty} \F_{\alpha}(u_j): u_j\weakly u \text{ in $\Spgm$} \}\\ 
& = \int_{\Omega} f^{\rm qc}(\na u)\dx+\int_{\Omega^c}f(\na u)\dx, \qquad u \in \Spgm,
\end{align*}
where $f^{\rm qc}:\Rmn \to \R$ denotes the quasiconvex envelope of $f$.
\end{theorem}

The rest of the paper is organized as follows. In Section~\ref{sec:preliminaries}, we begin by fixing notations and collect the necessary technical tools and auxiliary results from the theory of fractional Sobolev spaces, with a special view to the implications of a complementary-value condition, see especially~Lemma \ref{le:strongoutside}. 
Section~\ref{sec:connections} is then concerned with establishing the connections between the classical and fractional Sobolev spaces, and, in particular, between classical and fractional gradients (see Proposition~\ref{prop:connectionfraclaplacian}), which serve as the foundation for the proofs in the subsequent sections. The proof of the characterization result in Theorem \ref{th:characterization} is then presented in Section~\ref{sec:lsc}, along with a study of the new notion of $\alpha$-quasiconvexity. As a consequence, we also show the existence of minimizers to fractional variational problems on complementary-value spaces under suitable assumptions. Finally, we prove the relaxation result of Theorem \ref{theo:relaxation} in Section~\ref{sec:relaxation}.

\section{Preliminaries and technical tools}\label{sec:preliminaries}

We first introduce the notation and then give a selection of tools that will be used throughout the paper. 

\subsection{Notation} Unless mentioned otherwise, $\alpha\in (0,1)$ and $Q=(0,1)^n\subset \R^n$. We denote the Eulidean norm of a vector $x=(x_1, \ldots, x_n)\in \R^n$ by $\abs{x}=\left(\sum_{i=1}^n x_i^2\right)^{1/2}$ and similarly, the Frobenius norm of a matrix $A \in \Rmn$ by $\abs{A}$. The ball centered at $x \in \R^n$ and with radius $\rho>0$ is written as $B(x,\rho)=\{ y \in \R^n : \abs{x-y}<\rho\}$. 
For $E\subset \R^n$, we indicate its complement as $E^c:=\R^n \setminus E$ and its closure as $\overline{E}$. The notation $E \Subset F$ for sets $E, F\subset \R^n$ means that $E$ is compactly contained in $F$, i.e., $\overline{E} \subset F$ and $\overline{E}$ is compact. Let 
\[
\mathbbm{1}_E(x) = \begin{cases} 1 &\text{for} \ x \in E,\\
0 &\text{otherwise},
\end{cases}\qquad x\in \R^n,
\]
be the indicator function of a set $E \subset \R^n$.
Moreover, $\Gamma$ stands for Euler's gamma function.

Let $U\subset \R^n$ be an open set. 
The space $C_{c}^{\infty}(U)$ symbolizes the smooth functions $\varphi:U\to \R$ with compact support in $U \subset \R^n$.  Note that functions in $C_c^\infty(U)$ are identified with their trivial extension to $\R^n$ by zero without further mention. Further, let $C^\infty(\R^n)$ and $C_0(\R^n)$ be the spaces of smooth functions on $\R^n$ and continuous functions on $\R^n$ vanishing at infinity, respectively.

By $\Lipb(\R^n)$, we refer to all the functions $\psi:\R^n \to \R$ that are Lipschitz continuous and bounded on $\R^n$ and we write $\Lip(\psi)$ for the Lipschitz constant of $\psi$.
The space $C^{0,\beta}(\R^n)$ with $\beta\in (0,1]$ consists of all real-valued $\beta$-H\"{o}lder continuous functions defined on $\R^n$. 

The Lebesgue measure of $U \subset \R^n$ is denoted by $\abs{U}$, and for a 
 function $u:U\to \R$, let us introduce the notation
\[
\dashint_{U} u(x)\dx=\frac{1}{\abs{U}}\int_{U}u(x)\dx,
\]
provided the integral exists.
We use the standard notation for Lebesgue- and Sobolev-spaces, that is, $L^p(U)$ for $p\in [1, \infty]$ is the space of $p$-real-valued integrable functions on $U$ with the norm 
\[
\norm{u}_{L^p(U)}=\begin{cases}\displaystyle\left( \int_{U} \abs{u(x)}\dx\right)^{1/p} &\text{if} \ p \in [1,\infty),\\
\esssup_{x \in U} \abs{u(x)} &\text{if} \ p=\infty,
\end{cases}\qquad u\in L^p(U);
\]
for brevity, we write $\norm{u}_{L^{p}(\R^n)}=\norm{u}_p$ when $U=\R^n$.  
 Moreover, $W^{1,p}(U)$ for $p\in [1, \infty]$ is the space of $L^p$-functions on $U$ with $p$-integrable weak derivatives, endowed with the norm
\[
\norm{u}_{W^{1,p}(U)}=\norm{u}_{L^p(U)}+\norm{\nabla u}_{L^p(U)};
\]
here $\nabla u$ stands for the weak gradient of $u$. 

The spaces of functions that are locally in $L^p(\R^n)$ and $W^{1, p}(\R^n)$ are denoted by $L^p_{loc}(\R^n)$ and $W^{1,p}_{loc}(\R^n)$.
Besides, $W^{1,p}_{0}(U)$ stands for those functions in $W^{1,p}(U)$ with zero boundary value in the sense of the trace. By $C^{\infty}_{per}(Q)$ and $\Winfper$ we indicate the $Q$-periodic functions in $C^{\infty}(\R^n)$ and $W^{1,\infty}(\R^n)$, respectively.  Notice also that we make frequent use of the fact that $W^{1,\infty}(\R^n)=\Lipb(\R^n)$ as a consequence of Rademacher's theorem. Furthermore, $p'\in [1, \infty]$ stands for the dual exponent of $p$, i.e., $1/p+1/p'=1$, and we recall that a sequence $(v_j)_j\subset L^p(U)$ is called $p$-equi-integrable if $(|v_j|^p)_j$ is equi-integrable.  

In general, the definitions above can be extended componentwise to spaces of vector-valued functions. Our notation then explicitly mentions the target space, like, for example, $L^p(U;\R^m)$ consists of all functions $u:U\to \R^m$ whose individual components lie in $L^p(U)$.

Finally, we use $C$ to denote a generic constant, which may change from one estimate to the next without further mention. Whenever we wish to indicate the dependence of $C$ on certain quantities, we add them in brackets.

\subsection{Fractional Calculus}
\label{sec:rieszpotential}
We start with a few basic observations about a singular integral operator that plays a major role in the fractional calculus, namely the Riesz potential. 
\begin{definition}[Riesz potential]
For $u:\R^n \to \R^m$ measurable and $\alpha\in (0,n)$, the Riesz potential $I_{\alpha}u$ of $u$ of order $\alpha$ is defined as
\begin{align*}
I_{\alpha}u(x)=\frac{1}{\gamma_{n,\alpha}}\int_{\R^{n}} \frac{u(y)}{\abs{x-y}^{n-\alpha}}\dy, \qquad x\in \R^n,
\end{align*}
with $\gamma_{n,\alpha}=\pi^{n/2}2^{\alpha}\frac{\Gamma(\alpha/2)}{\Gamma((n-\alpha)/2)}$, provided the integrals exist a.e.~in $\R^n$. 
\end{definition}

Since the Riesz potential of a function $u$ corresponds to its convolution with the locally integrable $ x\mapsto \gamma_{n, \alpha}^{-1}\abs{x}^{\alpha-n}$, the function $u$ needs to decay to zero sufficiently fast in order for the Riesz potential of $u$ to be well-defined. 
 To be precise, $I_{\alpha}u$ is well-defined if and only if
\begin{equation}
\label{eq:rieszpotentialwelldefined}
\int_{\R^n}\frac{|u(y)|}{(1+\abs{y})^{n-\alpha}}\dy <\infty,
\end{equation}
see e.g.~\cite[Theorem 1.1, Chapter~2]{Mizuta}, and in this case, $I_{\alpha}u$ is locally integrable.
Note that \eqref{eq:rieszpotentialwelldefined} is satisfied for any $u \in L^p(\R^n)$ with $1\leq p<n/\alpha$ and  that, then $I_{\alpha}u \in L^p_{loc}(\R^n)$, see~\cite[Theorem 2.1, Chapter~4]{Mizuta}. 
 
Moreover, when $u \in L^{\infty}(\R^n)$ has compact support, then $I_{\alpha}u \in L^{\infty}(\R^n)$, and one obtains that $I_{\alpha}u\in L^{\infty}(\R^n)\cap C^\infty(\R^n)$ for any $\phi \in \Ccr$.
For more on the Riesz potential, there is a wide literature to refer to, such as~\cite{Mizuta, Stein}. 

The central objects of study in this manuscript involve Riesz fractional gradients, which can be defined  for bounded Lipschitz functions
as follows, cf.~\cite[Definition 2.1]{Silhavy},~\cite[Section~2.2]{Comi1}.
\begin{definition}[Riesz fractional gradient]
\label{def:rieszfractionalderivative}
Let $\alpha \in (0,1)$ and $\phi \in \Lipb(\R^n)$. Then the (vector-valued) Riesz fractional gradient of $\phi$ is defined as
\begin{align}\label{Rieszfracgradient}
\na \phi(x)=\mu_{n,\alpha}\int_{\R^n}\frac{\phi(y)-\phi(x)}{\abs{y-x}^{n+\alpha}}\frac{y-x}{\abs{y-x}}\dy,\qquad x\in \R^n,
\end{align}
with $\mu_{n,\alpha}=2^{\alpha}\pi^{-n/2}\frac{\Gamma((n+\alpha+1)/2)}{\Gamma((1-\alpha)/2)}$.
\end{definition}
The integral in~\eqref{Rieszfracgradient} with $\varphi\in \Lipb(\R^n)$ exists everywhere and $\na \phi \in L^{\infty}(\R^n;\R^n)$ with the bound
\begin{equation}\label{eq:interpolationbound}
\abs{\na \phi(x)} \leq |\mu_{n,\alpha}| \int_{\R^n} \frac{\abs{\phi(y)-\phi(x)}}{\abs{y-x}^{n+\alpha}}\dy \leq C(n,\alpha)\norm{\phi}_{\infty}^{1-\alpha}\Lip(\phi)^{\alpha} 
\end{equation} 
for all $x \in \R^n$, see~\cite[Lemma~2.3]{Comi2}. The previous definition extends naturally to $\phi \in \Lipb(\R^n;\R^m)$, with $\na \phi(x)$ elements of $\Rmn$, by taking the Riesz fractional gradient componentwise. 

We point out that, in contrast to the properties of classical gradients, considering $\phi \in \Ccr$ does not necessarily imply that also $\na \phi$ has compact support. In one dimension, for instance, a straightforward calculation shows that any non-negative $\phi \in C_c^{\infty}((0,1))$ satisfies $\na \phi(x) \not =0$ for $x\in (0,1)^c$. Still, if we follow \cite[Proposition 5.2]{Silhavy} in defining the class of functions
\[
\mathcal{T}(\R^n)=\{\psi \in C^{\infty}(\R^n)  :\partial^{a}\psi \in L^1(\R^n) \cap C_0(\R^n) \ \text{for all multi-indices} \ a \in (\N\cup\{0\})^n\}\subset \Lipb(\R^n),
\]
where $\partial^{a}\psi$ denotes the $a$th partial derivative of $\psi$,
then $\na \phi \in \mathcal{T}(\R^n;\R^n)$ for all $\phi \in \Ccr$. 
An alternative representation of the Riesz fractional gradient, with particular relevance for this paper, can be given in terms of a classical gradient and the Riesz potential, precisely,
\begin{equation}
\na \phi=\nabla I_{1-\alpha}\phi=I_{1-\alpha}\nabla \phi \qquad\text{for $\phi\in \Ccr$, }
\label{eq:rieszpotentialcharacterization}
\end{equation} 
see e.g.~\cite[Proposition 2.2]{Comi1}, \cite[Theorem~1.2]{Shieh1}; in Proposition~\ref{prop:connectionriesz}\,$(i)$ below  (cf.~also Remark~\ref{rem3}\,b)), we extend \eqref{eq:rieszpotentialcharacterization} to a more general class of functions.

As proven in~\cite[Theorem~2.2]{Silhavy}, the Riesz fractional gradient is (up to a constant) the only linear operator on $C_c^\infty(\R^n)$ that is rotationally and translationlly invariant, $\alpha$-homogeneous and satisfies weak requirement of continuity; here, the $\alpha$-homogeneity of $\na \varphi$ for $\varphi\in C_c^\infty(\R^n)$ means that for any $\lambda >0$,
\begin{equation*}\label{eq:homogeneity}
\na \phi_\lambda(x)=\lambda^{\alpha} \na \phi (\lambda x) \quad \text{for $x\in \R^n$ with $\varphi_\lambda = \varphi(\lambda\, \cdot\,)$}.
\end{equation*}
This characterization result identifies the Riesz fractional gradient in some sense as the canonical fractional derivative, although it has only received increased attention in recent times. A very common object in the literature, on the other hand, is the fractional Laplacian. Out of the number of equivalent definitions (cf.~\cite{Silhavy, Kwasnicki}), we choose to work here with the following. 
\begin{definition}[Fractional Laplacian]\label{def:fracLaplacian}
Let $\alpha \in (0,1)$ and $\phi \in \Lipb(\R^n)$. Then the (scalar-valued) fractional Laplacian of $\phi$ is defined as
\begin{align}\label{fracLaplacian}
\da{\alpha/2}\phi(x)= \nu_{n,\alpha}\int_{\Rn}\frac{\phi(x+h)-\phi(x)}{\abs{h}^{n+\alpha}}\, dh, \qquad x\in \R^n,
\end{align}
with $\nu_{n,\alpha}=2^{\alpha}\pi^{-n/2}\frac{\Gamma((n+\alpha)/2)}{\Gamma(-\alpha/2)}$.
\end{definition}
The fact that this object is well-defined and lies in $L^{\infty}(\R^n)$ can be seen similarly to \cite[Lemma 2.2]{Comi2}, 
and it holds that $\da{\alpha/2}\phi \in \mathcal{T}(\R^n)$ for $\phi \in \Ccr$ by \cite[Proposition 5.2]{Silhavy}.
For vector-valued functions, we let $(-\Delta)^{\alpha/2}$ act componentwise.

Both the Riesz fractional gradient as well as the fractional Laplacian satisfy a duality relation: for all $\phi \in \Ccr$ and $\psi \in \Lipb(\R^n)$, it holds that
\begin{equation}
\int_{\R^n} \na \phi\,\psi\dx=-\int_{\R^n}\phi\,\na \psi\dx
\label{eq:duality}
\end{equation}
and
\begin{equation}
\int_{\R^n} \da{\alpha/2} \phi\,\psi\dx=\int_{\R^n}\phi\,\da{\alpha/2} \psi\dx.
\label{eq:duality2}
\end{equation}
The identity~\eqref{eq:duality} follows from the equivalent formulation in \cite[Proposition 2.8]{Comi2}, which involves the fractional divergence, and~\eqref{eq:duality2} can be verified analogously. 

There are many useful composition rules for the fractional operators $\na,\da{\alpha/2}$ and the fractional divergence considered in \cite[Theorem 5.3]{Silhavy} which mirror the well-known relations from the classical case. For our purposes, there is one we wish to highlight, that is,
\begin{equation}\label{eq:composition}
\dac \na \phi=\na \dac \phi=\nabla \phi \quad \text{for $\phi \in \Ccr$;}
\end{equation}
strictly speaking, only the second identity is mentioned in \cite[Theorem 5.3]{Silhavy}, but in view of~ \eqref{eq:duality} and \eqref{eq:duality2}
we find that for any $\phi,\psi \in \Ccr$,
\[
\int_{\R^n} \dac \na \phi \,\psi \dx=\int_{\R^n} \na \phi \dac \psi \dx=-\int_{\R^n}\phi \na \dac \psi \dx=-\int_{\R^n}\phi \nabla \psi \dx,
\]
and thus, \eqref{eq:composition} by duality.

Moreover, there is a fractional analogue of the Leibniz rule available. As shown in \cite[Lemma 2.4]{Comi2}, it holds for $\phi \in \Ccr$ and $\psi \in \Lipb(\R^n)$ that
\begin{equation}\label{eq:leibniz}
\na (\psi \phi)=\psi\na\phi+\phi\na\psi+\nanl(\phi,\psi),
\end{equation}
with
\begin{align}\label{nablaNL}
\nanl(\phi,\psi)(x)=\mu_{n,\alpha}\int_{\R^n}\frac{(y-x)(\phi(y)-\phi(x))(\psi(y)-\psi(x))}{\abs{y-x}^{n+\alpha+1}}\dy, \quad x\in \R^n,
\end{align}
which is a pointwise well-defined function in $L^{\infty}(\R^n)$.
A similar result, stated in a different notation, can be found in \cite[Lemma 3.4]{Bellido}.

In preparation for an extension of the Leibniz rule to fractional Sobolev spaces in Section~\ref{sec:complementary}, we derive the following $L^p$-estimate with $p\in [1, \infty]$ for the expression in~\eqref{nablaNL}. 
Indeed, by Minkowski's integral inequality as in \cite[Section~A.1]{Stein} along with H\"{o}lder's inequality, and \eqref{eq:interpolationbound}, we obtain for $\phi \in \Ccr$ and $\psi \in \Lipb(\R^n)$ that
\begin{align}\label{eq:nanlbound}
\begin{split}
\norm{\nanl(\phi,\psi)}_p&\leq\abs{\mu_{n,\alpha}} \ \norm*{\int_{\R^n}\frac{\abs{\phi(\cdot+h)-\phi} \,\abs{\psi(\cdot+h)-\psi}}{\abs{h}^{n+\alpha}}\,dh}_p\\
&\leq \abs{\mu_{n,\alpha}}\int_{\R^n} \frac{\normb{\abs{\phi(\cdot+h)-\phi} \,\abs{\psi(\cdot+h)-\psi}}_p}{\abs{h}^{n+\alpha}}\,dh\\
&\leq 2\abs{\mu_{n,\alpha}}\norm{\phi}_p \int_{\R^n}\frac{\norm{\psi(\cdot+h)-\psi}_{\infty}}{\abs{h}^{n+\alpha}}\,dh\\
&\leq C(n,\alpha)\norm{\phi}_p\norm{\psi}_{\infty}^{1-\alpha}\Lip(\psi)^{\alpha}.
\end{split}
\end{align}

\subsection{Definition of fractional Sobolev spaces}
\label{sec:definition} 
Following the distributional approach in~\cite{Shieh1}, we introduce the fractional Sobolev spaces that are necessary for the definition and study of the fractional variational problems in this paper.   
For a more extensive treatment and further details, we refer to~\cite{Comi1, Comi2, Comi3, Shieh1, Shieh2} and the references therein.
Inspired by the fractional integration by parts formula (\ref{eq:duality}), we define a weak fractional gradient as below. In fact, this is equivalent to \cite[Definition 3.19]{Comi1}, which is formulated with the help of the fractional divergence.
\begin{definition}[Weak Riesz fractional gradient]\label{def:weakRiesz}
Let $\alpha \in (0,1)$ and $u\in L^1(\R^n) + L^\infty(\R^n)$. Then $v \in L^1_{loc}(\R^n;\R^n)$ is called the weak $\alpha$-fractional gradient of $u$ if
\begin{align}\label{wRieszgrad}
\int_{\R^n}v \phi\dx =-\int_{\R^n}u\na \phi \dx\qquad\text{for all} \  \phi \in \Ccr,
\end{align}
and we write $v=\na u$.
\end{definition}
We remark that the integral on the right-hand side of~\eqref{wRieszgrad} is well-defined, since $\na \phi \in \mathcal{T}(\R^n;\R^n)$, and that the weak fractional gradient is unique by the fundamental theorem in the calculus of variations. With this, one can now introduce the fractional Sobolev spaces.
\begin{definition}[Fractional Sobolev space]
Let $\alpha \in (0,1)$ and $p \in [1,\infty]$. The fractional Sobolev space $\Sp(\R^n)$ is the vector space of all functions in $L^p(\R^n)$ that have a weak $\alpha$-fractional gradient in $L^p(\R^n;\R^n)$, i.e.,
\[
\Sp(\R^n)=\{u \in L^p(\R^n)\,:\,\na u \in L^p(\R^n;\R^n)\}.
\]
This space is endowed with the norm
\[
\norm{u}_{\Sp(\R^n)}=\norm{u}_p+\norm{\na u}_p \qquad \text{for $u\in \Sp(\R^n)$.}
\]
By $\Sp(\R^n;\R^m)$, we denote the vector-valued analogue, obtained via a componentwise definition. 
\end{definition} 
 
The spaces $\Sp(\R^n)$ for $p \in (1,\infty)$ 
are in fact equivalent to the the fractional Sobolev spaces introduced earlier in \cite{Shieh1} as the closure of $C_c^{\infty}(\R^n)$ under the $\Sp(\R^n)$-norm. This follows directly from~\cite[Theorem A.1]{Comi3}, where it is proven that $\Ccr$ lies dense in $\Sp(\R^n)$, see also Theorem~\ref{th:density} below for an alternative proof.  Hence, as a consequence of the findings~in \cite{Shieh1}, the spaces $\Sp(\R^n)$ coincide with the well-known Bessel potential spaces, see e.g.~\cite{Adams}.

A comparison with the literature shows that the spaces denoted by $L^{s,p}(\R^n), X^{s,p}(\R^n)$ in \cite{Shieh1} and by $H^{s,p}(\R^n)$ in \cite{Shieh2,Bellido,Bellido2} all coincide with $\Sp(\R^n)$, if $s$ is replaced by $\alpha$. In particular, $\Sp(\R^n)$ inherits all the properties established in \cite{Shieh1,Shieh2}. One property we wish to mention here is the continuous embedding
\[
W^{1,p}(\R^n) \hookrightarrow \Sp(\R^n)
\]
for $\alpha \in (0,1)$ and $p\in (1, \infty)$.
Beyond that, we refer the interested reader to~\cite[Theorem 2.2~(g)]{Shieh1} and \cite[Proposition 3.24]{Comi1} for a discussion of how $\Sp(\R^n)$ relates to the more well-known fractional Sobolev spaces $W^{\alpha,p}(\R^n)$ involving the Gagliardo semi-norm; see \cite{Hitchhiker} for an elementary introduction to these spaces. 

The following fractional Poincar\'{e}-type inequality is proven~\cite[Theorem 3.3]{Shieh1}, see also~\cite{Shieh1} for a fractional Sobolev and Morrey inequality. 
\begin{theorem}
Let $\alpha \in (0,1)$, $p\in (1,\infty)$ and $\Omega \subset \R^n$ open and bounded. Then there exists a constant  $C=C(\Omega,n,p, \alpha)>0$ such that 
\[
\norm{u}_{L^p(\Omega)} \leq C\norm{\na u}_{p}
\]
for all $u \in \Sp(\R^n)$.
\label{th:poincare}
\end{theorem}

Next, we state the density result proven in \cite[Appendix A]{Comi3} and give an alternative and more elementary proof, cf.~also Remark~\ref{rem4}.

\begin{theorem}
Let $\alpha \in (0,1)$ and $p \in [1,\infty)$. Then $\Ccr$ is dense in $\Sp(\R^n)$.
\label{th:density}
\end{theorem}

\begin{proof}
Let $u \in \Sp(\R^n)$. In view of \cite[Theorem 3.22]{Comi1}, which states that $\Sp(\R^n)\cap C^{\infty}(\R^n)$ lies dense in $\Sp(\R^n)$, we may assume that $u \in \Sp(\R^n)\cap C^{\infty}(\R^n)$.
For each $j \in \N$, consider a cut-off function $\chi_j \in \Ccr$ such that 
\begin{align*}
\chi_j \equiv 1 \quad \text{ on $B(0,j)$ \quad $0 \leq \chi_j \leq 1$ \quad and \quad $\Lip(\chi_j) \leq 1/j$. }
\end{align*}

We will show that $\chi_j u \in \Ccr$ converges to $u$ in $\Sp(\R^n)$. It is immediate to see that $\chi_j u \to u$ in $L^p(\R^n)$. For the convergence of the fractional gradients, we note that $\chi_j u \in \Ccr\subset \Sp(\R^n)$ and calculate with the help of the Leibniz rule~\eqref{eq:leibniz} that 
\begin{align*}
\int_{\R^n} \na(\chi_j u)\phi \dx&=-\int_{\R^n}\chi_j u \na \phi\dx\\
&=-\int_{\R^n}u\na(\chi_j\phi)\dx+\int_{\R^n}u\phi \na \chi_j\dx+\int_{\R^n}u\nanl(\phi,\chi_j)\dx\\
&=\int_{\R^n} (\na u)\chi_j\phi\dx + \int_{\R^n}u\phi \na \chi_j\dx+\int_{\R^n}u\nanl(\phi,\chi_j)\dx
\end{align*} 
for any $\phi \in \Ccr$. Then, by H\"older's inequality, along with~\eqref{eq:interpolationbound} and \eqref{eq:nanlbound},
\begin{align*}
&\int_{\R^n} \abs{\na u-\na(\chi_ju)}\abs{\phi}\dx\\
&\qquad \qquad  \leq\int_{\R^n} \abs{(1-\chi_j)\na u}\abs{\phi}\dx+\int_{\R^n}\abs{u\na\chi_j}\abs{\phi}\dx+\int_{\R^n}\abs{u}\abs{\nanl(\phi,\chi_j)}\dx \\ 
&\qquad \qquad \leq  \norm{(1-\chi_j)\na u}_{p}\norm{\phi}_{p'} + C(n,\alpha)\norm{u}_{p}\norm{\chi_j}_{\infty}^{1-\alpha}\Lip(\chi_j)^{\alpha}\norm{\phi}_{p'},
\end{align*}
and since, in both terms, the factors in front of $\norm{\phi}_{p'}$ go to zero as $j \to \infty$, the convergence $\norm{\na u-\na(\chi_j u)}_p \to 0$ follows via duality.
\end{proof}
\begin{remark}\label{rem4}
The previous proof is similar to the arguments in~\cite[Theorem 3.23]{Comi1} for $p=1$, with the one difference  that the sequence of cut-off functions used there does not flatten gradually.
For the case $p \in (1,\infty)$ proven in~\cite[Theorem A.1]{Comi3}, the authors adopt a different approach, passing through the theory of Bessel potential spaces. 
\end{remark}
We use the following notion of weak convergence in $\Sp(\R^n)$, which corresponds to the abstract notion of weak convergence induced by the dual space.
\begin{definition}[Weak convergence]
Let $\alpha \in (0,1)$ and $p \in [1,\infty)$. We say that a sequence $(u_j)_j \subset \Sp(\R^n)$ converges weakly to $u$ in $\Sp(\R^n)$ if $u_j \weakto u$ in $L^p(\R^n)$ and $\na u_j \weakto \na u$ in $L^p(\R^n;\R^n)$, and we write $u_j \weakto u$ in $\Sp(\R^n)$.
\end{definition}

\subsection{Complementary-value fractional Sobolev spaces} \label{sec:complementary} The minimization problems we study here are formulated on fractional Sobolev spaces with fixed complementary values outside an open and bounded set;  precisely, for $\Omega \subset \R^n$ open and bounded, let 
\[
\Spo(\Omega)=\{ u \in \Sp(\R^n) \, :\,  u=0 \ \text{a.e. in} \ \Omega^c\},
\]
which is a closed subspace of $\Sp(\R^n)$, and define for given $g \in \Sp(\R^n)$ the affine subspace 
\begin{center}
$\Spg(\Omega)=g+\Spo(\Omega)$. 
\end{center}
When writing $u_j\weakly u$ in $S_g^{\alpha, p}(\Omega)$, we mean that $(u_j)_j \subset S_g^{\alpha, p}(\Omega)$ and $u\in S_g^{\alpha, p}(\Omega)$ such that $u_j\weakly u$ in $S^{\alpha, p}(\R^n)$.

We first recall the following compactness result on these spaces from \cite[Theorem 2.1]{Shieh2} and \cite[Theorem 2.3]{Bellido}, adjusted to our needs. It is crucial when applying the direct method in the calculus of variations to functionals of the form~\eqref{eq:functional1}.
\begin{theorem}
Let $\alpha \in (0,1)$, $p \in (1,\infty)$, $\Omega \subset \R^n$ open and bounded and $g \in \Sp(\R^n)$. Then, for any bounded sequence $(u_j)_j \subset \Spg(\Omega)$ there exists a subsequence (not relabeled) and a $u \in \Spg(\Omega)$ such that
\[
u_j \to u \ \text{in} \ L^{p}(\R^n) \qquad \text{and} \qquad \na u_j \weakto \na u \ \text{in} \ L^{p}(\R^n;\R^n). \
\]
\label{th:compactness}
\end{theorem}

A natural way of constructing functions in the complementary-value spaces involves cut-off arguments. To facilitate the latter, 
a generalization of the fractional Leibniz rule for functions in fractional Sobolev spaces $S^{\alpha, p}(\R^n)$ becomes necessary. The proof for $p \in [1,\infty)$ follows by density and can already be found in \cite[Lemma 3.4]{Bellido}. As the latter uses a different notation, we detail the argument below for the readers' convenience. The proof of the Leibniz rule for
$S^{\alpha, \infty}(\R^n)$ appears to be new in the literature.
\begin{lemma}
\label{le:leibnizp}
Let $\alpha \in (0,1)$, $p \in [1,\infty]$, $\psi \in \Lipb(\R^n)$ and $u \in \Sp(\R^n)$. Then $\psi u \in \Sp(\R^n)$ with
\begin{align}\label{chiu}
\na (\psi u)=\psi \na u+u \na \psi+\nanl(u,\psi),
\end{align}
and there is a constant $C=C(n,\alpha)>0$ such that
\begin{align}\label{est_cutoff}
\norm{\na (\psi u)-\psi \na u}_p
 \leq C\norm{\psi}_{\infty}^{1-\alpha}\Lip(\psi)^{\alpha}\norm{u}_p. 
\end{align}
\end{lemma}
\begin{proof} As indicated above, we prove the statement separately for $p<\infty$ and $p=\infty$.\medskip

\textit{Case $p\in [1, \infty)$.} Take  a sequence $(u_j)_j \subset \Ccr$ with $u_j \to u$ in $\Sp(\R^n)$ according to Theorem \ref{th:density}.  Then, $\psi u_j \in \Lipb(\R^n)$ and
\[
\na (\psi u_j)=\psi \na u_j+u_j \na \psi+\nanl(u_j,\psi)
\]
for $j\in \N$, in view of~\eqref{eq:leibniz}. 
Since $\psi u_j \to \psi u$ in $L^p(\R^n)$ and by \eqref{eq:nanlbound}
\[
\na (\psi u_j) \to \psi \na u+u \na \psi+\nanl(u,\psi) \ \ \text{in} \ L^p(\R^n;\R^n),
\]
we infer via the definition of the weak fractional gradient that $\na (\psi u)=\psi \na u+u \na \psi+\nanl(u,\psi)$, and hence, $\psi u \in \Sp(\R^n)$. The estimate~\eqref{est_cutoff} follows from~\eqref{eq:interpolationbound} and~\eqref{eq:nanlbound}. \medskip

\textit{Case $p=\infty$.} Identically to \eqref{eq:nanlbound}, we can show that $\nanl(u, \psi)$ is pointwise a.e.~well-defined and lies in $L^{\infty}(\R^n)$.
As we will show below,
\begin{equation} 
\int_{\R^n}\phi \nanl(u,\psi)\dx= \int_{\R^n}u \nanl(\phi,\psi)\dx
\label{eq:cancelintergal}
\end{equation}
for all $\phi \in \Ccr$, 
and hence, along with the Leibniz rule for smooth functions from~\eqref{eq:leibniz},
\begin{align*}
\int_{\R^n}u\psi \na \phi \dx &= \int_{\R^n}u\psi \na \phi + u \nanl(\phi,\psi)-\phi \nanl(u,\psi)\dx \\ &=\int_{\R^n} u\na (\psi \phi)-u \phi \na \psi-\phi\nanl(u,\psi)\dx  \\ & =- \int_{\R^n} (\psi \na u+u\na \psi+\nanl(u,\psi))\phi\dx,
\end{align*}
which gives~\eqref{chiu}. As above, the estimate~\eqref{est_cutoff} follows again from~\eqref{eq:interpolationbound} and~\eqref{eq:nanlbound}.

Finally,~\eqref{eq:cancelintergal} follows via Fubini's theorem and a symmetry argument. To be precise, one observes that
\begin{align*}
\begin{split}
\int_{\R^n}-u \nanl(\phi,\psi)\dx & =\int_{\R^n}\int_{\R^n} -u(x)\frac{(y-x)(\phi(y)-\phi(x))(\psi(y)-\psi(x))}{\abs{y-x}^{n+\alpha+1}}\dy\dx \\ & =
\int_{\R^n}\int_{\R^n} u(y)\frac{(y-x)(\phi(y)-\phi(x))(\psi(y)-\psi(x))}{\abs{y-x}^{n+\alpha+1}} \dy\dx,
\end{split}
\end{align*}
where the order of integration has been interchanged and the integration variables $x$ and $y$ have been renamed.
Then,
\begin{align*}
&\int_{\R^n}- u \nanl(\phi,\psi)+\phi \nanl(u,\psi)\dx \\ 
&\qquad\qquad \qquad =\int_{\R^n}\int_{\R^n} \frac{(y-x)(\psi(y)-\psi(x))}{\abs{y-x}^{n+\alpha+1}}(u(y)\phi(y)-u(x)\phi(x))\dy\dx=0, 
\end{align*}
exploiting that the double-integrand above is an odd function of the variables $x$ and $y$. This finishes the proof. 
\end{proof}
Interestingly, the following lemma shows that outside the domain the weak convergence of the fractional gradients is actually strong convergence. This observation appears to be new and has quite far-reaching consequences for the remainder of the paper.
\begin{lemma}
Let $\alpha \in (0,1)$, $p \in [1,\infty)$, $\Omega \subset \R^n$ open and bounded, and $g \in \Sp(\R^n)$. If $u_j \weakto u$ in $\Spg(\Omega)$, then 
\begin{align}\label{strongoncomplement}
\na u_j \to \na u \ \ \text{in} \ L^p((\Omega')^c;\R^n) 
\end{align}
for every open set $\Omega'\subset \R^n$ with $\Omega\Subset \Omega'$.
\label{le:strongoutside}
\end{lemma}
\begin{proof}
Due to linearity, it suffices to prove the statement for the special case $u=0$ and $g=0$. We consider therefore a sequence $(u_j)_j\subset S_0^{\alpha, p}(\Omega)$ with $u_j \weakto 0$ in $\Spo(\Omega)$. Let $\chi \in C_c^{\infty}(\Omega')$ be a cut-off function with $\chi \equiv 1$ on $\Omega$. Then, $u_j=\chi u_j$ for $j\in \N$ and $u_j \to 0$ in $L^p(\R^n)$ as a consequence of Theorem \ref{th:compactness}. Hence, we obtain from \eqref{est_cutoff} that
\begin{align*}
\norm{\na u_j}_{L^p((\Omega')^c;\R^n)}&=\norm{\na (\chi u_j)}_{L^p((\Omega')^c;\R^n)}\\
&\leq \norm{\na (\chi u_j) - \chi\na u_j}_p\leq C({\chi})\norm{u_j}_p \to 0\quad \text{as $j\to \infty$,}
\end{align*}
which yields~\eqref{strongoncomplement}. 
\end{proof}

\section{Connections between classical and fractional Sobolev spaces}\label{sec:connections}

In this section, we develop the primary tool for this paper, that is, a relation between the classical and fractional gradients in the context of Sobolev functions.
This connection relies on the properties of the Riesz potential and its inverse, the fractional Laplacian. Considering $\varphi\in C_c^\infty(\R^n)$, we know from~\eqref{eq:rieszpotentialcharacterization} and~\eqref{eq:composition} that
\begin{align}\label{connectionsCcinfty}
\nabla I_{1-\alpha} \varphi=\na \varphi\qquad \text{and}\qquad \na \dac \phi=\nabla \varphi.
\end{align} 
These identities enable us to express the fractional gradient of a function as the gradient of another function and vice versa. The goal is now to generalize this tool to the setting of Sobolev spaces. 

Before we present precise statements, the following short calculation gives some basic intuition. Suppose that 
$u\in L^1(\R^n)+L^{\infty}(\R^n)$ with $\na u\in L^1_{loc}(\R^n;\R^n)$
has a well-defined locally integrable Riesz potential $I_{1-\alpha}u$, cf.~\eqref{eq:rieszpotentialwelldefined}. Then, for any $\phi \in \Ccr$,
\begin{equation}
  \int_{\R^n} I_{1-\alpha} u\nabla \phi \dx=\int_{\R^n}uI_{1-\alpha}\nabla \phi \dx=\int_{\R^n}u\na \phi \dx = - \int_{\R^n}\na u \phi \dx,
\label{eq:fubini}
\end{equation}
where the first equality uses Fubini's theorem to shift the convolution to $\nabla \phi$, the second follows from~\eqref{eq:rieszpotentialcharacterization}, and the third is due to~Definition~\ref{def:weakRiesz}. This shows $\nabla I_{1-\alpha} u=\na u$, generalizing the first identity in~\eqref{connectionsCcinfty}. For the fractional Laplacian one uses a similar approach via duality to extend the identity in \eqref{connectionsCcinfty}.
The following proposition is similar in spirit to \cite[Lemma 3.28]{Comi1}, where instead of Sobolev spaces, the authors consider spaces with bounded fractional variation and establish a correspondence between classical and fractional variations.
\begin{proposition}
\label{prop:connectionriesz}
\label{prop:connectionfraclaplacian}
Let $\alpha \in (0,1)$ and $p \in [1,\infty]$. Then the following two statements hold: 
\begin{enumerate}
\item[$(i)$] For every $u \in \Sp(\R^n)$, there exists a $v \in W^{1,p}_{loc}(\R^n)$ such that 
$\nabla v= \na u$ on $\R^n$.
\item[$(ii)$]
For every $v\in W^{1,p}(\R^n)$, one has that 
\begin{center} $u=\dac v \in \Sp(\R^n)$\end{center} satisfies 
$\nabla^\alpha u= \nabla v$ on $\R^n$ and
\begin{align}\label{bound}
\norm{u}_p\leq C(n, \alpha)\norm{v}_p^{1-\frac{1-\alpha}{2}}\norm{\nabla v}_p^{\frac{1-\alpha}{2}}.
\end{align}
\end{enumerate}
\end{proposition}
\begin{proof}
\textit{Part $(i)$.} Let $u \in \Sp(\R^n)$. It is enough to find for every $R>0$ a $v \in W^{1,p}(B(0,R))$ such that $\nabla v=\na u$ on $B(0,R)$. We discuss the cases $p\in (1, \infty)$, $p=\infty$ and $p=1$ separately. 
In the first two cases, the idea is to approximate $u$ with functions that have a well-defined Riesz potential and subsequently use the calculation from \eqref{eq:fubini}. \medskip

\textit{Case $p\in (1, \infty)$.}  We approximate $u$ by a sequence $(u_j)_j \subset \Ccr$ in the $\Sp(\R^n)$-norm according to Theorem \ref{th:density} and define  for $j\in \N$, 
\begin{align*}
w_j=I_{1-\alpha}u_j \in L^\infty(\R^n)\cap C^\infty(\R^n),
\end{align*}
cf.~Section~\ref{sec:rieszpotential}. 
Moreover, it follows from (\ref{eq:rieszpotentialcharacterization}) that $w_j \in W^{1,p}(B(0,R))$ satisfies $\nabla w_j=\na u_j$. By subtracting the mean values, we define 
\[
v_j=w_j-\dashint_{B(0,R)}w_j\dy
\]
for $j\in \N$, and observe that $\nabla v_j=\na u_j$ and
\begin{align*}
\norm{v_j}_{L^p(B(0,R))} \leq C(R,n,p)\norm{\na u_j}_p 
\end{align*} 
due to the Poincar\'{e}-Wirtinger inequality. 
Hence, $(v_j)_j$ is bounded in $W^{1,p}(B(0,R))$, so that, up to taking a (non-relabeled) subsequence, $v_j \weakto v$ in $W^{1,p}(B(0,R))$ for some $v\in W^{1,p}(B(0, R))$. In combination with $\nabla v_j =\na u_j \to \na u$ in $L^p(B(0, R))$, we conclude that $\nabla v=\na u$ on $B(0,R)$.\medskip

\textit{Case $p=\infty$.} Consider a sequence of cut-off functions $(\chi_j)_j \subset \Ccr$ for $j\in \N$ such that
\[
\chi_{j}|_{B(0,R)}\equiv 1, \ \ 0\leq \chi_j\leq 1 \ \ \text{and} \ \ \Lip(\chi_j)\leq 1/j.
\]
Then, we infer from Lemma \ref{le:leibnizp} that $u_j:=\chi_ju \in \Sinf(\R^n)$ has compact support and satisfies
\begin{align}\label{est2}
\norm{\na u-\na u_j}_{L^{\infty}(B(0,R))}&\leq C\norm{\chi_j}_{\infty}^{1-\alpha}\Lip(\chi_j)^{\alpha}\norm{u}_{\infty}\leq C(1/j)^{\alpha}\norm{u}_{\infty}.
\end{align}
With $w_j:=I_{1-\alpha}u_j\in L^\infty(\R^n)$ for $j\in \N$,
we have by \eqref{eq:fubini} that $w_j \in W^{1,\infty}(B(0,R))$ with $\nabla w_j=\na u_j$ on $B(0,R)$. Now one can proceed similarly to the case $p\in (1, \infty)$ by defining
\[
v_j=w_j-\dashint_{B(0,R)}w_j\dy
\] 
for $j\in \N$ and by observing that the derivative of the weak$^\ast$ limit $v$ of $(v_j)_j\subset W^{1,\infty}(B(0,R))$ coincides with $\na u$ on $B(0, R)$, since $\na u_j \to \na u$ in $L^\infty(B(0,R))$ as $j\to \infty$ by~\eqref{est2}. \medskip

\textit{Case $p=1$.} This is a simple consequence of \eqref{eq:fubini}, cf.~Remark~\ref{rem3}\,b).\medskip 

\textit{Part (ii).} 
The fractional Laplacian as introduced in Definition~\ref{def:fracLaplacian} can be extended to a bounded linear operator $\dac :W^{1,p}(\R^n)\to L^p(\R^n)$ that satisfies the estimate
\begin{align}\label{bound}
\norm{\dac v}_p\leq C(n, \alpha)\norm{v}_p^{1-\frac{1-\alpha}{2}}\norm{\nabla v}_p^{\frac{1-\alpha}{2}}
\end{align}
for all $v\in W^{1,p}(\R^n)$. 
For $p\in [1, \infty)$, this is an immediate consequence of~\cite[Lemma A.4]{Comi3}, while in the case $p=\infty$, where $W^{1,\infty}(\R^n)=\Lipb(\R^n)$, we use the representation of $\dac$ in \eqref{fracLaplacian} and the bound~\eqref{bound} follows analogously to~\cite[Step~1, Lemma 2.2]{Comi2}.

   Setting $u=\dac v$, we obtain for any $\phi \in \Ccr$ that
\begin{align}\label{00}
\int_{\R^n} u \na \phi \dx=\int_{\R^n} v\, \dac \na \phi \dx 
=\int_{\R^n} v\nabla \phi \dx;
\end{align}
indeed, the first equality in the case $p \in (1, \infty)$ results from the duality of the fractional Laplacian in \eqref{eq:duality2} extended to pairs of functions $W^{1,p}(\R^n)$ and $W^{1,p'}(\R^n)$ via density (see \cite[Lemma A.5]{Comi3}), which again relies on the boundedness of $\dac$, and the observation that $\na \phi \in \mathcal{T}(\R^n;\R^n) \subset W^{1,p'}(\R^n;\R^n)$; for $p\in \{1,\infty\}$, it suffices to extend the duality \eqref{eq:duality2} to pairs in $\Lipb(\R^n)$ and $W^{1,1}(\R^n)$ in a similar way, owing to the observation $W^{1,\infty}(\R^n)=\Lipb(\R^n)$;
the second equality in~\eqref{00} uses the identity $\dac \na=\nabla$ on $\Ccr$ from \eqref{eq:composition}. This proves $\na u = \nabla v$ as stated.
\end{proof}

\begin{remark}\label{rem3}\leavevmode
a) Part $(ii)$ states in particular that $\dac:W^{1,p}(\R^n) \to S^{\alpha, p}(\R^n)$ is a bounded linear operator, and as such, weakly continuous. \medskip

b) The proof of part $(i)$ can be simplified in the special case when $u \in \Sp(\R^n)$ has a well-defined 
Riesz potential $I_{1-\alpha}u\in L^p_{loc}(\R^n)$ by setting $v=I_{1-\alpha}u$, which satisfies $\nabla v=\na u$ on $\R^n$ in light of~\eqref{eq:fubini}.
Recalling~\eqref{eq:rieszpotentialwelldefined}, we find this approach applicable in the regime $p <n/(1-\alpha)$, but it fails in general, as e.g.~$u(x)=\min\{1,\abs{x}^{-(1-\alpha)}\}$ for $x\in \R^n$ shows. Indeed, when $p>n/(1-\alpha)$ we find $u \in W^{1,p}(\R^n) \subset \Sp(\R^n)$, but $u$ does not have a well-defined Riesz potential by checking the criterion \eqref{eq:rieszpotentialwelldefined}. \medskip

c) We point out that Proposition~\ref{prop:connectionriesz}\,$(i)$ cannot be improved to finding for a given $u\in S^{\alpha, p}(\R^n)$ a function $v\in W^{1,p}(\R^n)$, so in particular, $v\in L^p(\R^n)$, with the stated properties.

For a counterexample in the case $p\in (1,\infty)$,
one can take, in view of a), any $u \in \Sp(\R^n)$ with a well-defined Riesz potential in $L_{loc}^p(\R^n)$ such that neither $I_{1-\alpha}u$, nor any translation of it, lies in $L^p(\R^n)$.
For example, it is straightforward to check that $u(x)=\min\{1,\abs{x}^{-n/p+\alpha-1}\}$ for $x\in \R^n$ satisfies $u\in W^{1,p}(\R^n)\subset S^{\alpha, p}(\R^n)$ and $I_{1-\alpha}u(x)\geq C(n, \alpha) |x|^{-n/p}$ for $x\in B(0,1)^c$.

For $p=\infty$, we pursue a different approach based on truncation to find functions $u \in \Sinf(\R^n)$ and $v \in W^{1,\infty}_{loc}(\R^n)\setminus L^\infty(\R^n)$ with $\nabla v=\na u$. The following construction is inspired by \cite[Lemma 3.1]{Bellido}.
Consider for fixed $\beta\in (0, 1-\alpha)$ the function
\[
v(x)=\begin{cases}
\abs{x} &\text{for $\abs{x} \leq 1$,}\\
\abs{x}^{\beta} &\text{for $\abs{x}>1$}, 
\end{cases}\qquad x\in \R^n,
\]
and let $v_{j}=\min\{v,j\}$ for $j\in \N$; by construction, $(v_j)_j \subset W^{1, \infty}(\R^n)\cap C^{0, \beta}(\R^n)$ is a sequence with uniformly bounded Lipschitz and $\beta$-H\"older constants. 
Hence, with $u_j:=\dac v_j \in \Sinf(\R^n)$, we obtain that
\begin{align*}
\abs{u_j}&\leq C(n, \alpha) \int_{\R^n} \frac{\abs{v_j(\cdot +h)-v_j}}{\abs{h}^{n+1-\alpha}}\,dh \\ &\leq C(n, \alpha)\Bigl(\int_{B(0,1)} \frac{1}{\abs{h}^{n-\alpha}}\,dh+\int_{B(0,1)^c} \frac{1}{\abs{h}^{n+1-\alpha-\beta}}\,dh\Bigr)<\infty
\end{align*}
for all $j\in \N$, where the Lipschitz and $\beta$-H\"older property of $v_j$ have been used to estimate the integrals over $B(0,1)$ and $B(0,1)^c$, respectively. Therefore, $(u_j)_j$ is a bounded sequence in $\Sinf(\R^n)$, so that $u_j \starto u$ in $L^{\infty}(\R^n)$ and
 $\na u_j \starto \na u$ in $L^{\infty}(\R^n;\R^n)$ as $j \to \infty$. Since $\na u_j=\nabla v_j$ for all $j\in \N$ due to Proposition~\ref{prop:connectionriesz}\,$(ii)$, we conclude that $\na u= \nabla v$, as required.
\end{remark}

Proposition~\ref{prop:connectionriesz} allows us to translate classical gradients of Sobolev functions into fractional gradients and the other way around, but it still leaves a small gap related to the issue of local integrability, cf.~Remark~\ref{rem3}\,c). 
If we focus on periodic Lipschitz functions, though, there is the following complete identification. 
\begin{proposition}
Let $\alpha \in (0,1)$ and $Q=(0,1)^n$. Then for every $v\in \Winfper$ there exists $u \in \Sinfper$ such that $\na u=\nabla v$ on $\R^n$ and vice versa.
\label{prop:periodicconnection}
\end{proposition}

\begin{proof}
First, let $v \in \Winfper\subset W^{1, \infty}(\R^n)$ be given. We set $u=\da{\frac{1-\alpha}{2}}v\in S^{\alpha, p}(\R^n)$, which, due to Proposition \ref{prop:connectionfraclaplacian}\,$(ii)$, satisfies the desired relation 
$\na u =\nabla v$ on $\R^n$. 
Given the periodicity of $v$, it is straightforward to verify in formula \eqref{fracLaplacian} for the fractional Laplacian of bounded Lipschitz maps that $u$ is $Q$-periodic, and thus, $u\in S^{\alpha, \infty}_{per}(Q)$.

Now, consider $u \in \Sinfper \subset S^{\alpha, \infty}(\R^n)$. By Proposition \ref{prop:connectionriesz}\,$(i)$, there is $v\in W^{1,\infty}_{loc}(\R^n)$ such that $\nabla v=\na u$ on $\R^n$. 
This implies, in particular, that the weak gradient of $v$ is $Q$-periodic on $\R^n$. To see that  $v$ itself is $Q$-periodic, we argue that, since any function in $W^{1,\infty}_{loc}(\R^n)$ with zero gradient is constant, there is a vector $a \in \R^n$ such that  
\begin{align}\label{55}
v(x+e_i)-v(x)=a_i\quad \text{ for all $x \in Q$ and $i=1,\dots, n$;}
\end{align}
here, $e_i$ denotes the $i$th standard unit vector in $\R^n$. The periodicity of $v$, follows, if one can show that $a=0$. 
To this end, we infer from the Gauss-Green theorem together with~\eqref{55} that
\begin{align}\label{11}
\int_{Q}\na u\dx=\int_{Q}\nabla v\dx=\int_{\partial Q}v\nu \, dS=a,
\end{align}
where $\nu$ is the outer unit normal vector field of the cube $Q$. 
By the $\alpha$-homogeneity of the fractional gradient, the sequence $(u_j)_j$ with $u_j(x)=\frac{1}{j^{\alpha}}u(jx)$ for $x\in \R^n$ and $j\in \N$ is bounded in $\Sinf(\R^n)$ and thus  up to subsequence $\na u_j$ converges weak* in $L^{\infty}(\R^n;\R^n)$. But we can calculate via the weak definition of the fractional gradient that for all $\phi \in \Ccr$
\[
\lim_{j \to \infty} \int_{\R^n} \na u_j \phi\dx=-\lim_{j \to \infty} \int_{\R^n} u_j \na \phi\dx =0
\]
since $u_j \to 0$ in $L^{\infty}(\R^n)$ so that
\begin{align}\label{22}
\na u_j\weaklystar 0 \quad \text{in} \  L^{\infty}(\R^n;\R^n).
\end{align} 
On the other hand, since $\nabla^\alpha u_j = \nabla^\alpha u(j\,\cdot)$ for $j\in \N$, the sequence $(\nabla^\alpha u_j)_j$ oscillates periodically and therefore satisfies 
\begin{align}\label{33}
 \na u_j\weaklystar \int_{Q}\na u\, dy \quad \text{in} \ L^{\infty}(\R^n;\R^n).
\end{align}
Combining~\eqref{22} and~\eqref{33} with~\eqref{11} finally gives $a= \int_{Q}\na u\, dy=0$. This shows that $v\in W^{1, \infty}_{per}(Q)$ and finishes the proof. 
\end{proof}

\section{Characterization of weak lower semicontinuity}\label{sec:lsc}

The main focus of this section is to provide a proof of Theorem~\ref{th:characterization}. This follows as a corollary of the combined statements of Theorem \ref{th:sufficientcomplementary} and \ref{th:necessarygeneral}, where we prove that quasiconvexity is necessary and sufficient for the weak lower semicontinuity of the functionals in~\eqref{eq:functional1}, respectively. The sufficiency follows rather quickly from well-known methods, using the tools from Section \ref{sec:connections} and the strong convergence of the fractional gradients outside $\Omega$ from Lemma \ref{le:strongoutside}. The necessary condition is more involved, as it requires the use of cut-off arguments and the careful construction of a function in the complementary-value space whose fractional gradient takes a prescribed value at a point.

We recall the definition of quasiconvexity (in the sense of Morrey \cite{Mor}). A Borel measurable function $h:\Rmn \to \R$ is called quasiconvex  if for any $A\in \Rmn$,
\begin{align}\label{qc1}
h(A)\leq \int_{Q} h(A+\nabla \phi)\dy \qquad\text{for all $\phi \in W^{1,\infty}_0(Q;\R^m)$,}
\end{align}
with $Q=(0,1)^n$.
Equivalently, by \cite[Proposition 5.13]{Dacorogna}, the class of test fields can be replaced with $Q$-periodic Lipschitz functions, that is, $h$ is quasiconvex if and only if for any $A\in \Rmn$,
\begin{align}\label{qc2}
h(A)\leq \int_{Q} h(A+\nabla \phi)\dy\qquad\text{for all $\phi \in  \Winfperm$.} 
\end{align}

The next theorem shows that the functionals in \eqref{eq:functional1} are weakly lower semicontinuous if the integrand $f$ is quasiconvex in its third variable. Note that we do not require quasiconvexity outside~$\Omega$. 
\begin{theorem}[Sufficiency of quasiconvexity]
\label{th:sufficientcomplementary}
Let $\alpha \in (0,1)$, $p \in (1,\infty)$, $\Omega \subset \R^n$ open and bounded and $g \in \Sp(\R^n;\R^m)$. Further, let 
\[
\F_{\alpha}(u)=\int_{\R^n} f(x,u(x),\na u(x))\dx, \qquad u \in \Spgm,
\]
where $f:\R^n\times\R^m\times \Rmn \to \R$ is a normal integrand satisfying 
\begin{align*}
0\leq f(x,z,A) \leq a(x)+ C(\abs{z}^p+\abs{A}^p) \quad \text{for a.e.~$x\in \R^n$ and for all} \ (z, A)\in \R^m\times \Rmn,
\end{align*}
with $a\in L^1(\R^n)$ and a constant $C>0$. If $A \mapsto f(x,z,A)$ is quasiconvex for a.e.~$x\in \Omega$ and all $z \in \R^m$
then the functional $\F_{\alpha}$ is (sequentially) weakly lower semicontinuous on $\Spgm$.
\end{theorem}

\begin{proof} We split $\F_{\alpha}$ into the integral over $\Omega$ and $\Omega^c$ and prove separately that
\begin{equation}\label{estsuff}
\int_{\Omega}f(x,u,\na u)\dx\leq \liminf_{j \to \infty} \int_{\Omega} f(x,u_j,\na u_j)\dx,
\end{equation}
and
\begin{equation}\label{estsuff2}
\int_{\Omega^c} f(x,u,\na  u)\dx \leq\liminf_{j \to \infty} \int_{\Omega^c}f(x,u_j,\na u_j)\dx,
\end{equation} 
for any sequence $(u_j)_j\subset \Spgm$ with $u_j \weakto u$ in $\Spgm$. 
In short, the idea for~\eqref{estsuff} is to reduce the problem to a classical weak lower semicontinuity result with the help of~Proposition \ref{prop:connectionriesz}\,$(i)$, while the key ingredient for~\eqref{estsuff2} is the strong convergence of the fractional gradients outside of $\Omega$, due to Lemma~\ref{le:strongoutside}. 

Let $(u_j)_j\subset \Spgm$ with $u_j \weakto u$ in $\Spgm$. Then also $u_j \to u$ in $L^p(\R^n)$ by the weak compactness in Theorem~\ref{th:compactness}. Moreover, Proposition \ref{prop:connectionriesz}\,$(i)$ allows us to find a sequence $(v_j)_j \subset W^{1,p}(\Omega;\R^m)$ and $v \in W^{1,p}(\Omega;\R^m)$ such that 
\begin{align}\label{vjuj}
\nabla v_j=\na u_j \text{ on $\Omega$ for all $j\in \N$} \quad  \text{and}\quad  \nabla v =\na u \text{ on $\Omega$.}
\end{align} 
By assuming that the functions $v_j$ and $v$ have mean value zero, we infer along with Poincar\'e's inequality that $(v_j)_j$ is a bounded sequence in $W^{1,p}(\Omega;\R^m)$, and hence, $v_j \weakto v$ in $W^{1,p}(\Omega;\R^m)$. 

The remaining argument uses standard elements of Young measure theory, for a general introduction to Young measures see e.g.~\cite[Chapter 8]{Fonseca},~\cite{Pedregal},~\cite[Chapter 4]{Rindler}. After passing to subsequences (not relabeled), the gradients $(\nabla v_j)_j$ generate a Young measure $\{\mu_x\}_{x \in \Omega}$, and $((u_j,\nabla v_j))_j\subset L^p(\Omega)\times L^p(\Omega;\R^{m\times n})$ generates the parametrized product measure $\{\delta_{u(x)} \otimes \mu_x\}_{x \in \Omega}$, where $\delta_z$ denotes the Dirac measure centered in $z\in \R^m$, cf.~e.g.~\cite[Corollary 8.10]{Fonseca}. Then the fundamental theorem of Young measures gives
\begin{align*}
\liminf_{j \to \infty} \int_{\Omega} f(x,u_j,\nabla v_j)\dx &\geq \int_{\Omega}\int_{\Rmn}f(x,u,A)\,d{\mu_x(A)}\dx. 
\end{align*} 
Based on the well-known characterization of gradient Young measures in \cite{KiP91, KiP94}, see also~\cite[Theorem 7.15]{Rindler},
we can now exploit the quasiconvexity of $f$ in its third argument with $p$-growth to obtain the Jensen-type inequality 
 \begin{align*}
\int_{\R^{m\times n}} f(x, u(x), A) \, d\mu_x(A) \geq f\Bigl(x, u(x),  \int_{\R^{m\times n}} A\, d\mu_x(A)\Bigr)= f(x, u(x),\nabla v(x)) \quad
\end{align*}
 for a.e.~$x\in \Omega$, 
where we have used that the barycenter of $\mu_x$ corresponds to $\nabla v(x)$. In view of~\eqref{vjuj} we see that \eqref{estsuff} follows.

For any $\Omega'\subset \R^n$ with $\Omega \Subset \Omega'$, we have by Lemma \ref{le:strongoutside} that $\na u_j \to \na u$ in $L^p((\Omega')^c;\Rmn)$. Hence, by a standard lower semicontinuity result for strong $L^p$-convergence (see~\cite[Theorem 6.49]{Fonseca}),
\begin{align*} 
\int_{(\Omega')^c} f(x,u,\na  u)\dx &\leq \liminf_{j \to \infty} \int_{(\Omega')^c}f(x,u_j,\na u_j)\dx\\ & \leq\liminf_{j \to \infty} \int_{\Omega^c}f(x,u_j,\na u_j)\dx,
\end{align*}
so that by letting $\Omega'$ tend to $\Omega$, the monotone convergence theorem in combination with the non-negativity of $f$ implies \eqref{estsuff2}. Finally, the stated weak lower semicontinuity of $\F_{\alpha}$ follows from the combination of~\eqref{estsuff} and~\eqref{estsuff2}.
\end{proof}

\begin{remark}
The bounds 
on the integrand function $f$ in Theorem~\ref{th:sufficientcomplementary} can be weakened,
with slightly different conditions on $\Omega$ and its complement. On $\Omega$, we may require as a  lower bound that for some $q\in [1, p)$,
\[
-a(x)-C(\abs{z}^p+\abs{A}^q) \leq f(x,z,A) \quad \text{for a.e.~$x\in \Omega$ and for all} \ (z, A)\in \R^m\times \Rmn
\]
with $a \in L^1(\Omega)$ and a constant $C>0$; 
indeed, this ensures that the negative part of the sequence $\bigl(f(\cdot ,u_j,\na u_j)\bigr)_j$ is equi-integrable, given the convergence of $(u_j)_j$ in $L^p(\Omega;\R^m)$ and the boundedness of $\na u_j$ in $L^p(\Omega;\Rmn)$, and
 hence, the fundamental theorem of Young measures applies in the same manner.
On $\Omega^c$, we may lose the upper bound entirely,
and one can use the lower bound
\[
-b(x)-C(\abs{z}^p+\abs{A}^p) \leq f(x,z,A) \quad \text{for a.e.~$x\in \Omega^c$ and for all} \ (z, A)\in \R^m\times \Rmn
\]
with $b \in L^1(\Omega^c)$ and $C>0$, considering that \cite[Theorem 6.49]{Fonseca} holds under this assumption.
\end{remark}
In preparation for the necessary condition, we construct functions with compact support whose fractional gradients take a certain given value at a point. This tool serves us as a replacement for the locally affine functions, which are commonly employed in the classical theory.

\begin{lemma}
\label{le:construction}
Let $\alpha \in (0,1)$ and let $\Omega \subset \R^n$ be open and bounded. For any $x_0\in \Omega$, $z \in \R^m$ and $A \in \Rmn$, there exists a $\phi \in C_{c}^{\infty}(\Omega;\R^m)$ such that $\phi(x_0)=z$ and $\na \phi(x_0)=A$.
\end{lemma}

\begin{proof}
Due to the translation invariance and $\alpha$-homogeneity of the fractional gradient, one may assume without loss of generality that $\Omega=B(0,1)$ and $x_0=0$. Further, we may suppose that $z=0$, 
since any radially symmetric function in $C^\infty_c(B(0,1))$ has a vanishing fractional gradient at the origin and can therefore be added afterwards to guarantee that $\varphi(0)=z$.  Arguing componentwise, we only need to consider the case $m=1$, and by the linearity of the fractional gradient it is enough to construct $\phi\in C_c^\infty(B(0,1))$ with $\phi(0)=0$ and $\na \phi(0)$ a multiple of any of the standard unit vectors $e_1, \ldots, e_n$ in $\R^n$. 

We pick $e_1\in \R^n$ (the argumentation for the unit vectors $e_2, \ldots, e_n$ is analogous), and let  $\theta, \psi \in C_{c}^{\infty}((-1,1))$ be even and odd functions, respectively, i.e., 
\begin{center}
$\theta(t)=\theta(-t)$ and $\psi(-t)=-\psi(t)$ for $t\in (-1,1)$, 
\end{center}
such that $\theta, \psi$ are supported in the interval $(-\frac{1}{2}, \frac{1}{2})$, are not identically zero and are non-negative on $(0,1)$. Now, if we define $\phi \in C_c^{\infty}(B(0,1))$ by
\[
\phi(x)=\psi(x_1)\theta(x_2)\dots\theta(x_n) \quad\text{for $x=(x_1, x_2, \ldots, x_n)\in B(0,1)$,}
\]
then $\phi(0)=0$ and by exploiting the symmetry properties of $\varphi$, 
\begin{align*}
(\na \phi(0))_1=\mu_{n,\alpha}\int_{\R^n}\frac{\phi(y)y_1}{\abs{y}^{n+\alpha+1}}\dy=2\mu_{n,\alpha}\int_{\{y\in \R^n:y_1>0\}}\frac{\phi(y)y_1}{\abs{y}^{n+\alpha+1}}\dy =: \beta >0. 
\end{align*}
On the other hand, $(\na \phi(0))_i=0$ for $i=2, \ldots, n$, since the integrand is odd with respect to the $i$th variable. Hence, we have shown that $\nabla^\alpha \varphi(0)=\beta e_1\neq 0$, as desired.
\end{proof}

\begin{remark}\label{re:construction}
The previous lemma can be extended to functions with more general complementary values. Let $\alpha\in (0, 1)$, $\Omega \subset \R^n$ open and bounded, and $g \in \Sp(\R^n;\R^m)$ with $p\in (1, \infty)$.  Then, for a.e.~$x_0 \in \Omega$ there exists a $u \in \Spgm$ such that $x_0$ is a $p$-Lebesgue point of $u$ and $\na u$ with $u(x_0)=z$ and $\na u(x_0)=A$, that is,
\[
\lim_{\rho \to 0} \rho^{-n}\int_{B(x_0,\rho)}\abs{u(x)-z}^p\dx=0 \ \ \text{and} \ \ \lim_{\rho \to 0} \rho^{-n}\int_{B(x_0,\rho)}\abs{\na u(x)-A}^p\dx=0.
\]

This follows immediately from Lemma \ref{le:construction} if we take $x_0$ to be a $p$-Lebesgue point of $g$ and $\na g$, and set $u=g+\phi$ with a function $\phi \in C_{c}^{\infty}(\Omega;\R^m)$ such that
\[
\phi(x_0)=z-g(x_0) \ \ \text{and} \ \ \na \phi(x_0)=A-\na g(x_0).
\]
\end{remark}

Now, we can complement the sufficiency statement of Theorem~\ref{th:sufficientcomplementary} by proving that quasiconvexity of the integrand is also necessary in order for $\F_{\alpha}$ as in~\eqref{eq:functional1} to be lower semicontinuous. The proof combines our new insights about the connections between gradients and fractional gradients from Section~\ref{sec:connections} with some well-established techniques in the classical calculus of variations, see e.g.~\cite{AcerbiFusco,Dacorogna}.

\begin{theorem}[Necessity of quasiconvexity]
\label{th:necessarygeneral}
Let $\alpha \in (0,1)$, $p \in (1,\infty)$, $\Omega \subset \R^n$ open and bounded, and $g \in \Sp(\R^n;\R^m)$. Suppose that $f:\R^n\times \R^m\times\Rmn \to \R$ is a Carath\'{e}odory function satisfying
\begin{align}\label{growth_ness}
\abs{f(x,z,A)} \leq a(x)+C(\abs{z}^p+\abs{A}^p) \quad \text{for a.e.~$x\in \R^n$ and for all $(z, A)\in \R^m\times \R^{m\times n}$}
\end{align}
with $a \in L^1(\R^n)$ and $C>0$. If
\[
\F_{\alpha}(u)=\int_{\R^n}f(x,u(x),\na u(x))\dx, \qquad  u\in\Spgm,
\]
is (sequentially) weakly lower semicontinuous on $\Spgm$, then $A \mapsto f(x,z,A)$ is quasiconvex for a.e. $x \in \Omega$ and all $z \in \R^m$.
\end{theorem}
\begin{proof}
First, we observe that it suffices to prove the statement for the case of vanishing complementary value, that is, for $g=0$. Indeed, the weak lower semicontinuity of $\F_{\alpha}$ on $\Spgm$ is equivalent to the weak lower semicontinuity on $\Spom$ of the auxiliary functional
\[
\F_{\alpha,g}(u):=\int_{\R^n} f_g(x,u,\na u)\dx, \qquad u\in \Spom,
\]
where $f_g(x,z,A)=f(x,z+g(x),A+\na g(x))$ for a.e.~$x\in \R^n$ and for $(z, A)\in \R^m\times \Rmn$. One can check that $f_g$ is also a Carath\'{e}odory integrand that satisfies the $p$-growth condition \eqref{growth_ness}, and that 
$f(x,z,\cdot)$ is quasiconvex if and only if $f_{g}(x,z-g(x),\cdot)$ is quasiconvex.\medskip

To show that $f(x,z,\cdot)$ is quasiconvex for a.e.~$x \in \Omega$ and all $z \in \R^m$, let $(x_0, z_0, A_0)\in \Omega\times \R^m\times \Rmn$, and take a $u\in C_c^\infty(\Omega;\R^m)$ such that 
\begin{center}
$u(x_0)=z_0$ \quad and \quad  $\nabla^\alpha u(x_0)=A_0$, 
\end{center} 
according to Lemma~\ref{le:construction}. Now, we fix $\varphi\in W^{1, \infty}_0(Q;\R^m)$.  Under consideration of the translation and scaling invariance with respect to the domain of the test fields in the definition of quasiconvexity, we may assume without loss of generality that $x_0+Q\Subset \Omega$. We split the remaining proof into four steps. The second part of the proof is a modification of \cite[Lemma 3.18]{Dacorogna}, where
the key issue is to substitute affine functions by functions in the complementary-value space whose fractional gradient only attains the desired value at a single point, as constructed in Lemma~\ref{le:construction}.
\medskip

\textit{Step 1. }
Let $\rho\in (0,1)$. After extending $\phi$ periodically to $\R^n$, consider the sequence $(\phirho_j)_j \subset W^{1,p}(\R^n;\R^m)$ given by
\[
\phirho_j(x)=\begin{cases}
\displaystyle \frac{\rho}{j}\phi\Bigl(j\frac{(x-x_0)}{\rho}\Bigr) &\text{for $x \in Q_{\rho}:= x_0+ 
(0,\rho)^n$,}\\
0 &\text{otherwise, }
\end{cases} \qquad x\in \R^n.
\]
As $(\phirho_j)_j$ is bounded in $W^{1,p}(\R^n;\R^m)$ and converges uniformly to zero, we observe that
\begin{align}\label{weakphirho}
\phirho_j \weakly 0\quad \text{ in $W^{1,p}(\R^n;\R^m)$,}
\end{align}
as $j \to \infty$. Applying Proposition \ref{prop:connectionfraclaplacian}\,$(ii)$ shows for $j\in \N$ that 
\begin{align*}
\phirhotilde_j:=\dac \phirho_j \in \Sp(\R^n;\R^m) \end{align*}
satisfies $\na {\phirhotilde}_j=\nabla \phirho_j$ on $\R^n$. Moreover, we have in light of~\eqref{bound} and~\eqref{weakphirho} that the sequence $(\phirhotilde_j)_j$ is bounded in $\Sp(\R^n;\R^m)$ and satisfies in the limit $j\to \infty$ that 
\begin{align*}
{\phirhotilde}_j \weakto 0\quad \text{ in $\Sp(\R^n;\R^m)$} \qquad \text{and} \qquad {\phirhotilde}_j \to 0\quad\text{ in $L^p(\R^n;\R^m)$.}
\end{align*} 
If we take a cut-off function $\chi \in C_{c}^{\infty}(\Omega)$ with $\chi \equiv 1$ on $x_0+Q$, it follows from Lemma \ref{le:leibnizp} that, as $j\to \infty$,
\begin{align*}
u_j:=u+\chi {\phirhotilde}_j \weakto u\quad \text{in $S^{\alpha, p}_0(\Omega;\R^m)$}  
\end{align*}
and
\[
R_j^\rho:=\na u_j-\na u-\chi \na \phirhotilde_j \to 0 \qquad \text{in $L^p(\R^n;\Rmn)$. }
\]
Besides, we observe that $\chi\na \phirhotilde_j=\nabla \phirho_j$ on $\R^n$, as $\nabla\phirho_j=0$ outside of $Q_{\rho}$.  

Then, by the weak lower semicontinuity of $\F_{\alpha}$ on $S^{\alpha, p}_0(\Omega;\R^m)$ and with the help of Lemma \ref{le:freezing}  and Lebesgue's dominated convergence theorem,
\begin{align*}
& \int_{\R^n}f(x,u,\na u)\dx \leq \liminf_{j \to \infty}\int_{\R^n} f(x,u_j,\na u_j)\dx\\
&\qquad = \liminf_{j \to \infty} \int_{Q_{\rho}} f(x,u+\phirhotilde_j,\na u+\nabla \phirho_j+R_j^\rho)\dx+\int_{Q_{\rho}^c}f(x,u+\chi{\phirhotilde}_j,\na u+R_j^\rho)\dx\\
&\qquad \leq \liminf_{j \to \infty} \int_{Q_{\rho}} f(x,u,\na u+\nabla \phirho_j)\dx+\int_{Q_{\rho}^c}f(x,u,\na u)\dx.
\end{align*}
If we subtract from both sides of the previous estimate the integral over $Q_{\rho}^c$, which is finite by the $p$-growth of $f$, we find
\begin{equation}
\int_{Q_{\rho}}f(x,u,\na u)\dx\leq \liminf_{j \to \infty} \int_{Q_{\rho}} f(x,u,\na u+\nabla \phirho_j)\dx.
\label{eq:cubeinequality}
\end{equation}
\medskip

\textit{Step 2. } Denoting
\[
\lambda:=\norm{u}_{L^{\infty}(\Omega;\R^n)}+\norm{\na u}_{L^{\infty}(\Omega;\Rmn)}+\norm{\nabla \phi}_{L^{\infty}(Q;\Rmn)}<\infty,
\]
we introduce the compact set
\[
S=\{(z,A) \in \R^m \times \Rmn \ | \  \abs{z}+\abs{A} \leq \lambda\}.
\]
By the Scorza-Dragoni theorem (see e.g.~\cite[Theorem~6.35]{Fonseca}), there exists a sequence of compact nested sets $K_{l} \subset \Omega$  for $l \in \N$ such that $f$ is continuous on $K_{l} \times  S$ and $\abs{\Omega \setminus K_{l}} \leq 1/l$, and Tietze's extension theorem allows us to find for each $l\in \N$ a continuous function $f_l:\Omega\times \R^m\times \Rmn\to \R$ coinciding with $f$ on $K_{l} \times S$ such that
\[
\abs{f_{l}(x,z,A)} \leq M_l:=\max\{ f(x,z,A) \ | \ (x,z,A) \in K_{l}\times S\}
\] 
for all~$x\in \Omega$ and all $(z,A) \in \R^m \times\Rmn$.  

One can even arrange for any fixed $\epsilon' >0$ that for all $l\in \N$,
\begin{equation}
\int_{\Omega \setminus K_{l}}f_{l}(x,w(x),W(x))\dx \leq \epsilon'
\label{eq:extensionbound}
\end{equation}
for any $w \in L^{\infty}(\Omega;\R^m), W \in L^{\infty}(\Omega;\Rmn)$. Indeed, this can be done by replacing $f_{l}$ by $\eta_{l} f_{l}$ with $\eta_l \in C_{c}^{\infty}(\Omega)$ such that $\eta_l \equiv 1$ on $K_{l}$, $0\leq \eta_l \leq 1$ and
\[
\int_{\Omega \setminus K_{l}} \eta_l (x)\dx \leq \frac{\epsilon'}{M_l}.
\]

Now, suppose that $x_0$ lies in $K:=\bigcup_{l\in \N} K_{l}$ and is a Lebesgue point of the indicator functions $\mathbbm{1}_{\Omega\setminus K_{l}}$ and $a \mathbbm{1}_{\Omega \setminus K_{l}}$ for all $l \in \N$, recalling that $a\in L^1(\R^n)$ from~\eqref{growth_ness}. We denote the set of all such points by $\Omega_0\subset \Omega$ and observe that the set $\Omega\setminus \Omega_0$ has zero $n$-dimensional Lebesgue-measure. 
If $l\geq l_{0}:=\min\{l\in \N: x_0\in K_l\}$, it holds that
\begin{align*}
&\lim_{\rho \to 0} \rho^{-n} \int_{Q_{\rho} \setminus K_{l}} f(x,u(x),\na u(x)+\nabla \phirho_j(x))\dx  \leq \lim_{\rho \to 0} \rho^{-n} \int_{Q_{\rho} \setminus K_{l}} a(x)+2C\lambda^p \dx \\ &\qquad \qquad \qquad = \lim_{\rho\to 0} \dashint_{Q_\rho} a\mathbbm{1}_{\Omega \setminus K_l} \, dx + 2C\lambda^p \lim_{\rho\to 0} \dashint_{Q_\rho} \mathbbm{1}_{\Omega\setminus K_{l}}\, dx=0,
\end{align*}
where the last identity exploits that $\mathbbm{1}_{\Omega\setminus K_l}(x_0)=0$ and $(a\mathbbm{1}_{\Omega\setminus K_l})(x_0)=0$. 

Hence, with $\eps>0$ given, we have shown that for $x_0\in \Omega_0$,
\begin{equation}
\rho^{-n}\int_{Q_{\rho} \setminus K_{l}} f(x,u,\na u+\nabla \phirho_j)\dx < \epsilon
\label{eq:smallsetbound}
\end{equation} 
for all $l\geq l_{0}$, all $j\in \N$ and all $\rho>0$ sufficiently small.
\medskip

\textit{Step 3.} Returning to \eqref{eq:cubeinequality}, we claim that for any $x_0\in \Omega_0$, for every $l\geq l_0$, $j\in \N$ and $\rho>0$ small enough that
\begin{align*}
& \Bigl|  \int_{Q_{\rho}} f(x,u,\na u+\nabla \phirho_j)\dx - \int_{Q_{\rho}}f_{l}(x_0,z_0,A_0+\nabla \phirho_j)\dx \Bigr| \\  & \qquad\qquad\leq
\int_{Q_{\rho}}\abs{f_{l}(x_0,z_0,A_0+\nabla \phirho_j)-f_{l}(x,u,\na u+\nabla \phirho_j)}\dx\\
&\qquad\qquad\qquad\qquad  + \int_{Q_{\rho}} \abs{f_{l}(x,u,\na u+\nabla \phirho_j)-f(x,u,\na u+\nabla \phirho_j)}\dx\\ 
&\qquad\qquad  =: \ (A)+(B) < 2\eps \rho^n + \eps';
\end{align*}
indeed, since $(u(x),\na u(x)+\nabla \phirho_j(x)) \in S$ for a.e.~$x \in \Omega$, the estimate 
\[
(B) \leq \int_{Q_{\rho}\setminus K_{l}} \abs{f_{l}(x,u,\na u+\nabla \phirho_j)-f(x,u,\na u+\nabla \phirho_j)}\dx\leq \epsilon \rho^n +\epsilon',
\]
follows from \eqref{eq:extensionbound} and \eqref{eq:smallsetbound}; for the bound on (A), we use the uniform continuity of $f_{l}$ on compact sets and the continuity of $u$ and $\na u$ to deduce $(A) \leq \epsilon \rho^n$. 

Lastly, we conclude from the periodicity of $\phirho_j$ and the property that $x_0 \in K_{l}$ that
\begin{align*}
\int_{Q_{\rho}}f_{l}(x_0,z_0,A_0+\nabla \phirho_j)\dx=\rho^n\int_{Q}f_{l}(x_0,z_0,A_0+\nabla \phi)\dy=\rho^n\int_{Q}f(x_0,z_0,A_0+\nabla \phi)\dy. 
\end{align*}

\textit{Step~4.} All in all, we obtain in view of Step~2 and Step~3 that for $x_0\in \Omega_0$, so, in particular, for a.e.~$x_0\in \Omega$, and $\rho>0$ sufficiently small, 
\begin{align*}
\liminf_{j \to \infty} \int_{Q_{\rho}} f(x,u,\na u+\nabla \phi_j)\dx&\leq \rho^n \int_{Q} f(x_0,z_0,A_0+\nabla \phi)\dy + 2\epsilon\rho^n+\epsilon'\\
& \leq \rho^n\int_{Q} f(x_0,z_0,A_0+\nabla \phi)\dy  + 3\epsilon\rho^n,
\end{align*}
by choosing $\epsilon'\leq \epsilon/\rho^n$, and similar, even simpler argument, gives
\begin{align*}
\int_{Q_{\rho}}f(x,u,\na u)\dx\geq \rho^n f(x_0,z_0,A_0)-3\epsilon \rho^n .
\end{align*}
Hence, for a.e.~$x_0\in \Omega$, \eqref{eq:cubeinequality} becomes after division by $\rho^{n}$,
\[
f(x_0,z_0,A_0) \leq \int_{Q}f(x_0,z_0,A_0+\nabla \phi)\dy+6\epsilon,
\]
which by the arbitrariness of $\epsilon$ and $A_0\in \Rmn$, shows the quasiconvexity of $f(x_0, z_0, \cdot)$ and thus, finalizes the proof. 
\end{proof}

It may seem surprising at first that the weak lower semicontinuity of functionals involving fractional gradients can be characterized via a notion of convexity that is intrinsic to the classical setting. To enlarge upon this point, we introduce and analyze the following new type of convexity, which appears to be the natural translation of quasiconvexity into the fractional context.
\begin{definition}[{$\alpha$}-quasiconvexity]\label{def:alphaqc}
Let $\alpha \in (0,1)$. We say that a Borel measurable function $h: \Rmn \to \R$ is $\alpha$-quasiconvex if it holds for any $A \in \Rmn$ that
\[
h(A)\leq \int_{Q}h(A+\na \phi)\dy \qquad \text{ for all $\phi \in S^{\alpha,\infty}_{per}(Q;\R^m)$.}
\]
\end{definition}

Let us briefly comment on a few aspects regarding alternative definitions of $\alpha$-quasiconvexity.

\begin{remark}
a) Notice that, in analogy to quasiconvexity, one could replace the unit cube $Q$ in the definition of $\alpha$-quasiconvexity through a translation and scaling argument with any 
other open cube $Q'\subset \R^n$ and test with $Q'$-periodic functions in $S^{\alpha, \infty}_{per}(Q';\R^m)$ instead.\medskip

b) If $h:\Rmn\to \R$ is continuous, we observe that $h$ is $\alpha$-quasiconvexity if and only if
\begin{equation*}
h(A)\leq \int_{Q}h(A+\na \phi)\dy \quad\text{for all $\phi \in C^{\infty}_{per}(Q;\R^m)$}
\end{equation*}
for any $A\in \Rmn$.
This follows from Lebesgue's dominated convergence theorem, since every $\phi \in \Sinfperm$ can be approximated via standard mollification by a sequence $(\phi_j)_j \subset C^{\infty}_{per}(Q;\R^m)$ in the sense that $\na \phi_j \to \na \phi$ in $L^1(Q;\Rmn)$ and $\norm{\na \phi_j}_{\infty} \leq \norm{\na \phi}_{\infty}$ for all $j\in \N$, cf.~also \cite[Theorem 3.22]{Comi1}.   \medskip

c) As already mentioned above, it is well-known that quasiconvexity can be equivalently defined via test fields $W^{1,\infty}_0(Q;\R^m)$ or $\Winfperm$, cf.~\eqref{qc1} and~\eqref{qc2}. Therefore, one might wonder whether for a Borel measurable function $h:\Rmn\to \R$ the property that for any $A\in \Rmn$,
\begin{equation}
h(A)\leq \int_{Q}h(A+\na \phi)\dy \quad\text{for all $\phi \in S^{\alpha,\infty}_{0}(Q;\R^m)$,}
\label{eq:badquasiconvexity}
\end{equation}
constitutes a meaningful alternative generalization of quasiconvexity to the fractional setting.
However, this is not the case, since this notion
fails to generalize classical convexity even for $n=m=1$.
In fact, we show below that one can find $\phi \in C_{c}^{\infty}((0,1))$ with
\begin{equation}\label{eq:nonzeromean}
(\na \phi)_{(0,1)}:=\int_0^1 \na \phi\dy\not= 0,
\end{equation}
which implies, in particular, that no linear function $h:\R \to \R$ with $h((\na\phi)_{(0,1)}) <0$ satisfies~\eqref{eq:badquasiconvexity}, because
\[
h(0)=0>h((\na\phi)_{(0,1)})=\int_{0}^1h(\na \phi)\dy.
\]
To see \eqref{eq:nonzeromean}, choose $\tilde{\phi} \in C_c^{\infty}((0,1))$ to be non-negative and not identically zero. Then, $\na \tilde{\phi}(0)>0$, so that, by the continuity of $\na \tilde{\phi}$, there is some $\delta >0$ with
\[
\int_{-\delta}^1 \na \tilde{\phi}\dy \not =0.
\]
Transforming $\tilde \varphi$ on the interval $(-\delta, 1)$ to a function $\varphi\in C_c^\infty((0,1))$, under consideration of the translation invariance and $\alpha$-homogeneity of the fractional gradient, yields~\eqref{eq:nonzeromean}.  
\end{remark}

In fact, the newly established $\alpha$-quasiconvexity
is equivalent to the classical quasiconvexity for any $\alpha\in (0,1)$ as an immediate consequence of Proposition \ref{prop:periodicconnection}, and  provides therefore, another way to characterize the weak lower semicontinuity of the functionals $\F_{\alpha}$ in~\eqref{eq:functional1}.
\begin{corollary}\label{cor:equivalence}
Let $\alpha \in (0,1)$ and $h:\Rmn \to \R$ a Borel function. Then, $h$ is quasiconvex if and only if $h$ is $\alpha$-quasiconvex.
\end{corollary}
With the previous results at hand, a straightforward application of the direct method in the calculus of variations provides the following existence result for integral functionals involving fractional gradients.
\begin{theorem}[Existence of minimizers]
Suppose the setting of Theorem \ref{th:sufficientcomplementary}, in particular, $\R^{m\times n}\ni A\mapsto f(x, z, A)$ is quasiconvex for a.e.~$x\in \Omega$ and all $z\in \R^m$, and assume further that the Carath\'eodory integrand $f$ satisfies the coercivity condition
\[
c\abs{A}^p-b(x) \leq f(x,z,A) \qquad \text{for a.e.~$x\in \R^n$ and $(z,A) \in\times\R^m\times \Rmn$}
\]
with a constant $c>0$ and $b \in L^1(\R^n)$. Then, the functional $\F_{\alpha}$ admits a minimizer in $\Spgm$.
\end{theorem}
\begin{proof}
Let $(u_j)_j \subset \Spgm$ be a minimizing sequence of $\F_{\alpha}$, then the coercivity condition yields that $\sup_{j\in \N}\norm{\na u_j}_p  <\infty$ for all $j\in \N$. Along with the fractional Poincar\'{e} inequality in Theorem \ref{th:poincare}, we infer that $(u_j)_j$ is a bounded sequence in $\Spgm$, so that by Theorem~\ref{th:compactness}, up to the selection of a non-relabeled subsequence, $u_j \weakto u$ in $\Spgm$ for some $u\in \Spgm$. The weak lower semicontinuity of $\F_{\alpha}$ from Theorem \ref{th:sufficientcomplementary} finally shows that $u$ is indeed a minimizer of $\F_{\alpha}$ over $\Spgm$.
\end{proof}
To close this section, we state and briefly prove for the readers' convenience the following well-known tool, which has been used  
in the proof of Theorem~\ref{th:necessarygeneral} and will be exploited once again in the next section. 
\begin{lemma}
\label{le:freezing}
Let $U \subset \R^n$ open and bounded and $f:U \times \R^m \times \Rmn \to \R$ be Carath\'eodory and
\[
\abs{f(x,z, A)} \leq a(x) + C(\abs{z}^p+\abs{A}^p) \quad \text{for a.e.~$x\in U$ and for all $(z, A)\in \R^m\times \Rmn$}
\]
with $a \in L^1(U)$ and a constant $C>0$. If $u_j \to u$ in $L^p(U;\R^m)$, $v_j \to v$ in $L^p(U;\Rmn)$ and $(w_j)_j$ is a bounded sequence in $L^p(U;\Rmn)$ that is $p$-equi-integrable, then
\[
\liminf_{j \to \infty} \int_{U} f(x,u_j,w_j+v_j)\dx\leq\liminf_{j \to \infty} \int_{U} f(x,u,w_j+v)\dx.
\]
\end{lemma}
\begin{proof}
By choosing subsequences (not relabeled), we may assume that
\[
\liminf_{j\to \infty}\int_{U} f(x,u,w_j+v)\dx=\lim_{j \to \infty} \int_{U} f(x,u,w_j+v)\dx,
\]
that $(w_j)_j$ generates a Young measure $\{\mu_x\}_{x \in U}$, and  that $u_j \to u$ pointwise a.e.~in $U$. Under consideration of  the $p$-equi-integrability of $(w_j)_j$, it follows then from the fundamental theorem for Young measures (see e.g.~\cite[Theorem 8.6\,(ii), Corollary~8.7, Theorem~8.10]{Fonseca}) that
\[
\lim_{j \to \infty}\int_{U}f(x,u_j,w_j+v_j)=\int_{U}\int_{\Rmn}f(x,u,A+v)\,d{\mu_x(A)}\dx=\lim_{j \to \infty}\int_{U} f(x,u,w_j+v)\dx.
\]
\end{proof}

\section{Relaxation}\label{sec:relaxation}
Finally, we establish the relaxation formula
stated in Theorem \ref{theo:relaxation}. For simplicity, we drop the explicit $x$- and $u$-dependence of the integrand, which, at the same time, helps to illustrate the
structural change of the integral functional induced by the relaxation process, namely from a homogeneous to an inhomogeneous integrand.
The idea of the proof is to reduce to a classical relaxation result by exploiting the tools developed in Section \ref{sec:connections}.

Recall that the relaxation of $\F_{\alpha}$ as in \eqref{eq:integralfunctional2} refers to the largest weakly lower semicontinuous functional that lies below $\F_{\alpha}$, 
which can be represented as
\[
\Frel(u)= \inf \{ \liminf_{j \to \infty} \F_{\alpha}(u_j) \ | \ u_j \weakto u \ \text{in} \ \Spgm\} \quad\text{for $u \in \Spgm$}.
\]
According to the results in Section \ref{sec:lsc}, the weak lower semicontinuity of functionals as in \eqref{eq:integralfunctional2} is equivalent to the quasiconvexity of the integrand $f$, which makes quasiconvexification a natural ingredient in the relaxation process. 
The quasiconvex envelope of $f$, i.e., the largest quasiconvex function below $f$, is given by
\[
f^{\rm qc}(A)=\sup \{g(A)  :  g:\Rmn \to \R \ \text{quasiconvex}, \ g \leq f\} ,
\]
or equivalently, characterized as
\[
f^{\rm qc}(A)=\inf_{\phi \in  W^{1,\infty}_0(Q;\R^m)}\int_{Q}f(A+\nabla \phi)\dy
\] 
for $A \in \Rmn$,
see \cite[Theorem 6.9]{Dacorogna}, \cite[Section 7.1]{Rindler}. 
We now proceed with the proof of the relaxation result.

\begin{proof}[Proof of Theorem \ref{theo:relaxation}]
In the following, let
\begin{align*}
\tilde{\F}_{\alpha}(u)=\int_{\Omega} f^{\rm qc}(\na u)\dx+\int_{\Omega^c}f(\na u)\dx = \int_{\R^n}\tilde f(x,\na u)\dx, \quad u\in S^{\alpha,p}_g(\Omega;\R^m),
\end{align*}
with $\tilde f: \R^n\times \R^{m\times n}\to \R$ given by $\tilde f(x, A)=\mathbbm{1}_{\Omega}(x) f^{\rm qc}(A)+ \mathbbm{1}_{\Omega^c}(x) f(A)$ for $(x, A)\in \R^n\times \R^{m\times n}$. We observe that, since the quasiconvexification $f^{\rm qc}$ is continuous and inherits the growth assumptions of $f$, the functional $\tilde{\F}_{\alpha}$ satisfies the requirements of Theorem~\ref{th:sufficientcomplementary} and is therefore weakly lower semicontinuous on $\Spgm$. In view of $\tilde{\F}_{\alpha} \leq \F_{\alpha}$, we therefore conclude $\tilde{\F}_{\alpha} \leq \Frel$.\medskip

To prove the converse inequality, take $u \in \Spgm$. By Proposition~\ref{prop:connectionriesz}\,$(i)$, there is a $v \in W^{1,p}(\Omega;\R^m)$ such that $\nabla v= \na u$ on $\Omega$. According to a classical relaxation result (see e.g.~\cite[Theorem 7.6]{Rindler}, \cite[Theorem 9.1]{Dacorogna}) we find a sequence $(v_j)_j\subset W^{1,p}(\Omega;\R^m)$ with $v_j=v$ on $\partial \Omega$ in the sense of traces for all $j\in \N$ such that 
$v_j \weakto v$ in $W^{1,p}(\Omega;\R^m)$ 
and
\begin{align}\label{classicalrelax}
\liminf_{j \to \infty} \int_{\Omega} f(\nabla v_j)\dx \leq \int_{\Omega} f^{\rm qc}(\nabla v)\dx.
\end{align}
Further, a well-known decomposition lemma (see e.g.~\cite[Lemma 8.15]{Dacorogna} or \cite[Lemma 11.4.1]{AMB14}) allows us to assume that the sequence $(\nabla v_j)_j$ is $p$-equi-integrable; a cut-off argument in the spirit of~\cite[Lemma 4.13, Step 3]{Rindler} shows that the boundary conditions can be preserved.

For each $j\in \N$, we extend $v_j-v$ trivially by zero outside of $\Omega$ to a function in $W^{1,p}(\R^n;\R^m)$ and 
define 
\begin{align*}
\tilde{u}_j:=\dac (v_j-v) \in \Sp(\R^n;\R^m). 
\end{align*} 
Observe that the sequence $(v_j-v)_j$ converges weakly to zero in $W^{1,p}(\R^n;\R^m)$. Therefore, we deduce from  Proposition~\ref{prop:connectionfraclaplacian}\,$(ii)$ that the sequence $(\tilde u_j)_j$ is 
bounded in $\Spmn$ and that \begin{align}\label{conv2}
\tilde{u}_j \weakto 0\quad \text{ in $\Spmn$, }
\end{align}
considering Remark~\ref{rem3}~a).
Moreover, since $v_j-v \to 0$ in $L^p(\R^n;\R^m)$, the estimate~\eqref{bound} yields
\begin{align}\label{strong1}
\tilde{u}_j \to 0 \quad \text{in $L^p(\R^n;\R^m)$.}
\end{align}

Now, let $O\subset \R^n$ be compactly contained in $\Omega$ and let  $\chi \in C_c^{\infty}(\Omega)$ be a cut-off function with $0\leq \chi \leq 1$ and $\chi \equiv 1$ on $O$. We consider the sequence $(u_j)_j\subset \Spgm$ given by 
\begin{align*}
u_j:=u+\chi\tilde{u}_j,\quad j\in\N.
\end{align*} 
In light of~\eqref{strong1}, Lemma \ref{le:leibnizp} implies
\begin{align}\label{Sj}
R_j:= \nabla^\alpha u_j -\nabla^\alpha u  -\chi \nabla^\alpha \tilde u_j\to 0 \ \ \text{in} \ L^p(\R^n;\Rmn).
\end{align}
Hence, together with~\eqref{conv2}, the sequence $(u_j)_j$ converges weakly to $u$ in $S^{\alpha, p}_g(\Omega;\R^m)$.

Using Lemma \ref{le:freezing} in combination with the $p$-equi-integrability of $(\nabla v_j)_j$ and the observation that $R_j=\nabla^\alpha u_j-\nabla v_j$ on $O$ for all $j\in \N$ results in
\begin{align}\label{est90}
\liminf_{j \to \infty} \int_{O} f(\nabla v_j)\dx \geq  \liminf_{j \to \infty} \int_{O} f(\nabla v_j+ R_j)\dx=\liminf_{j \to \infty} \int_{O} f(\na u_j)\dx.
\end{align}
Moreover, as $\na u_j - \na u\to 0$ in $L^p(\Omega^c;\R^m)$ again by~\eqref{Sj}, it follows from the continuity of $f$ and the dominated convergence theorem that
\begin{align}\label{est99}
\lim_{j \to \infty} \int_{\Omega^c} f(\na u_j)\dx =\int_{\Omega^c} f(\na u)\dx.
\end{align}
Lastly, let $\eps>0$. Then, we exploit the growth condition on $f$, to obtain \begin{align}\label{est87}
\int_{\Omega \setminus O}f(\na u_j)\dx \leq C\int_{\Omega \setminus O} \abs{(1-\chi)\nabla v+\chi\nabla v_j+R_j}^p\dx\leq \epsilon
\end{align}
for all $j\in \N$, if $|\Omega\setminus O|$ is sufficiently small. Here, once again, the $p$-equi-integrability of $(\nabla v_j)_j$ has been used. 

Summing up, we obtain from~\eqref{est90},~\eqref{est99} and~\eqref{est87} in combination with~\eqref{classicalrelax} that
\begin{align*}
\liminf_{j \to \infty}\F_{\alpha}(u_j)&\leq \liminf_{j \to \infty}\int_{O} f(\nabla v_j)\dx+\epsilon+\int_{\Omega^c}f(\na u)\dx \\ & \leq \int_{\Omega}f^{\rm qc}(\nabla v)\dx+\epsilon+\int_{\Omega^c}f(\na u)\dx \\  & =\int_{\Omega}f^{\rm qc}(\na u)\dx+\int_{\Omega^c}f(\na u)\dx + \eps=\tilde{\F}_{\alpha}(u)+\epsilon.
\end{align*}
The statement follows if we take the limit $\abs{\Omega\setminus O} \to 0$, and thus, $\epsilon \to 0$. 
\end{proof}

\begin{remark}
Alternatively, the statement about the necessary condition in Theorem~\ref{th:necessarygeneral} can be derived from 
Theorem ~\ref{theo:relaxation}, in the special case when the $x$- and $z$-dependence of $f$ is dropped and the additional coercivity assumption $f(A) \geq c\abs{A}^p$ for all $A\in \R^{m\times n}$ with a constant $c>0$ holds. 

Observe that if the functional $\F_{\alpha}$ is (sequentially) weakly lower semicontinuous on $\Spgm$, then it needs to coincide with its relaxation, that is, $\F_{\alpha}(u)=\Frel(u)$ for all $u\in \Spgm$.  
We subtract the identical integral terms over the complement of $\Omega$, which are finite due to the $p$-growth of $f$, to conclude that
\begin{align*}
\int_{\Omega}(f-f^{\rm qc})(\na u)\dx=0 \quad \text{for all $u\in \Spgm$.}
\end{align*}
Now, for given $A\in \Rmn$, we find by Remark \ref{re:construction} a specific $u\in \Spgm$ and a $p$-Lebesgue point $x_0\in \Omega$ of $\na u$ such that $\na u(x_0) = A$. 
Then, since $f^{\rm qc} \leq f$, it holds for any $\rho>0$ with $B(x_0,\rho) \subset \Omega$ that
\[
\int_{B(x_0,\rho)} (f-f^{\rm qc})(\na u)\dx=0,
\]
and after multiplication by $\rho^{-n}$, 
\begin{equation}\label{eq:comparison}
0=\rho^{-n}\int_{B(x_0,\rho)} (f-f^{\rm qc})(\na u)\dx=\int_{B(x_0,1)} (f-f^{\rm qc})(\na u(\rho x))\dx.
\end{equation}
Recalling that $x_0$ is a $p$-Lebesgue point of $\na u$ implies $\na u(\rho\,\cdot\,) \to A$ in $L^p(B(x_0,1);\Rmn)$ as $\rho \to 0$, and hence,  the right hand side of \eqref{eq:comparison} converges to $(f-f^{\rm qc})(A)$ as $\rho \to 0$, using the $p$-growth and continuity of $f$ and $f^{\rm qc}$ in combination with Lebesgue's dominated convergence theorem. This shows that $f=f^{\rm qc}$, and therefore, that $f$ is quasiconvex.
\end{remark}

\section*{Acknowledgements}
The authors would like to thank Javier Cueto for helpful comments on a preliminary version of the manuscript.
Part of this research was done while CK was affiliated with Utrecht University; in particular, CK acknowledges partial support by a Westerdijk Fellowship.

\addcontentsline{toc}{section}{\protect\numberline{}References}
\bibliographystyle{abbrv}
\bibliography{ThesisBib}
\end{document}